\documentclass[12pt,reqno]{amsart}
\usepackage{amsmath,amsthm,amssymb,amsfonts,amscd}
\usepackage{mathrsfs,cite}
\usepackage{bbm}
\usepackage{bbding}

\usepackage{bm}
\usepackage[backref=page]{hyperref}

\usepackage{geometry}
\usepackage{color}
\usepackage{xcolor}
\usepackage{enumerate}
\usepackage{picture,epic}
\usepackage{tikz}

\numberwithin{equation}{section}

\theoremstyle{plain}
\newtheorem{theorem}{Theorem}[section]
\newtheorem{lemma}[theorem]{Lemma}

\newtheorem{proposition}[theorem]{Proposition}

\theoremstyle{definition}

\theoremstyle{remark}
\newtheorem{remark}[theorem]{Remark}

\renewcommand{\Re}{\operatorname{Re}}
\renewcommand{\Im}{\operatorname{Im}}

\newcommand{\GL}{\operatorname{GL}}
\newcommand{\SL}{\operatorname{SL}}

\newcommand{\dd}{\mathrm{d}}

   \DeclareFontFamily{U}{wncy}{}
    \DeclareFontShape{U}{wncy}{m}{n}{<->wncyr10}{}
    \DeclareSymbolFont{mcy}{U}{wncy}{m}{n}
    \DeclareMathSymbol{\Sh}{\mathord}{mcy}{"58}

\makeatletter
\def\@tocline#1#2#3#4#5#6#7{\relax
  \ifnum #1>\c@tocdepth 
  \else
    \par \addpenalty\@secpenalty\addvspace{#2}%
    \begingroup \hyphenpenalty\@M
    \@ifempty{#4}{%
      \@tempdima\csname r@tocindent\number#1\endcsname\relax
    }{%
      \@tempdima#4\relax
    }%
    \parindent\z@ \leftskip#3\relax \advance\leftskip\@tempdima\relax
    \rightskip\@pnumwidth plus4em \parfillskip-\@pnumwidth
    #5\leavevmode\hskip-\@tempdima
      \ifcase #1
       \or\or \hskip 1em \or \hskip 2em \else \hskip 3em \fi%
      #6\nobreak\relax
    \hfill\hbox to\@pnumwidth{\@tocpagenum{#7}}\par
    \nobreak
    \endgroup
  \fi}
\makeatother

\begin{document}

\title[Extreme central values]{Extreme central values of quadratic Dirichlet $L$-functions with prime conductors}

\author{Mingyue Fan, Shenghao Hua, Sizhe Xie}

\address{Data Science Institute and School of Mathematics \\ Shandong University \\ Jinan \\ Shandong 250100 \\China}
\email{myfan@mail.sdu.edu.cn}
\email{huashenghao@mail.sdu.edu.cn}
\email{szxie@mail.sdu.edu.cn}

\date{\today}

\begin{abstract}
  In this paper we prove a lower bound result for extremely large values of $L(\frac{1}{2},\chi_p)$ with prime numbers $p\equiv 1\pmod 8$, based on generalizing the twisted first moment result of Baluyot--Pratt to the case of short Dirichlet series with large coefficients.
\end{abstract}


\thanks{M. F. was partially supported by the National Key R\&D Program of China (No. 2021YFA1000701).
S. H. was partially supported by the National Key R\&D Program of China (No. 2021YFA1000700) and NSFC
(No. 12031008).
All authors were partially supported by the Shandong Provincial Excellent Youth Science Fund Project (Overseas) (2022HWYQ-046).}


\maketitle

\section{Introduction} \label{sec:Intr}

The study of central values of $L$-functions related to quadratic characters is an interesting topic (see e.g. \cite{BP22,DGH03,DW21,GZ23,GH85,HS22,HH22a,
HH22b,HH22c,Ju81,RS15,Sh22,Sono20,So00,
So08,So21,SY10,Yo09,Yo13}).
Any real primitive character with the modulus $d$ must be of the form $\chi_d(\cdot)=(\frac{d}{\cdot})$, where $d$ is a fundamental discriminant \cite[Theorem 9.13]{MV07}, i.e., a product of pairwise coprime factors of the form $-4$, $\pm 8$ and $(-1)^{\frac{p-1}{2}}p$, where $p$ is an odd prime.
Knowledge about moments of central $L$-values is often the basis for the study of the distribution of central $L$-values.
Jutila\cite{Ju81} and Vinogradov--Takhtadzhyan\cite{VT81} independently established asymptotic formulas for the first moment of quadratic Dirichlet $L$-functions, whose conductors are all fundamental discriminants, or all prime numbers.
Afterwards, Goldfeld--Hoffstein\cite{GH85} and Young\cite{Yo09} improved the error term.
There also are many developments in setting up asymptotic formulas of higher moments.
Jutila\cite{Ju81} also obtained an asymptotic formula for the second moment, and Sono\cite{Sono20} improved the error term of it.
Soundararajan\cite{So00} got an asymptotic formula for the third moment, and he proved at least $87.5\%$ of the odd square-free integers $d\geq 0$, $L(\frac{1}{2},\chi_{8d})\neq 0$, whereas $L(\frac{1}{2},\chi)\neq 0$ holds for all primitive quadratic characters is a well-known conjecture of Chowla\cite{Ch65}.
Diaconu--Goldfeld--Hoffstein\cite{DGH03}, Young\cite{Yo13}, and Diaconu--Whitehead\cite{DW21} improved the error term for the third moment.

For a random variable which satisfies a known distribution, we may guess the maximum value that will occur.
Then we may can guess the extremum of central $L$-values based on Keating--Snaith conjecture\cite{KS00b}.
Farmer--Gonek--Hughes\cite{FGH07} set up this conjecture for different families of $L$-functions above.
The extreme values of the quadratic Dirichlet $L$-functions at the central point has been studied by Heath--Brown (unpublished, see\cite{HL99}) and Soundararajan\cite{So08}, and recently, Darbar and Maiti\cite{DM24,DM25} have obtained some conditional results that yield higher-order powers.

Soundararajan's resonance method is applicable to a wide range of families of $L$-functions.
One of its characteristics is the twisted Dirichlet series have large coefficients.
Recently, Baluyot and Pratt\cite{BP22} proved that $L(\frac{1}{2},\chi_{p})\neq 0$ for more than $8\%$ of the primes $p\equiv 1\pmod 8$, and the twisted short Dirichlet series they need are satisfy a condition like the generalized Ramanujan conjecture  non-vanishing problem.
In this paper, we demonstrate that Baluyot--Pratt's strategy can be extended to a wider class of twists, including those with large Dirichlet coefficients. Specifically, we demonstrate that when diagonal terms exhibit rapid growth, they could dominate the off-diagonal contributions within certain ranges.
Then we could obtain a result of extremely large central $L$-values in this family.

\begin{theorem}\label{thm1}
For sufficiently large $X$ we have
\begin{equation*}
  \max_{\substack{X< p\leq 2X\\\textrm{prime }p\equiv 1\pmod 8}}
  L(\frac{1}{2},\chi_p)
  \geq
\exp\Big((\sqrt{\frac{8}{45}}+o(1))\frac{\sqrt{\log X}}{\sqrt{\log\log X}}\Big),
\end{equation*}
where the constant in the little-$o$ term can be determined effectively.
\end{theorem}

\begin{remark}
  Here we used a lower bound result instead of an asymptotic formula for twisted first moment, and to get a minimum result would much more difficult.
\end{remark}

\begin{remark}
For the extremum of $\zeta(s)$ in the critical line, one can get better results in the logarithmic sense. Bondarenko and Seip\cite{BS17} obtained a substantial improvement that
\begin{equation*}
  \max_{0<t\leq T}|\zeta(\frac{1}{2}+it)|
  \geq \exp\Big(
  (c+o(1))\sqrt{\frac{\log T \log\log\log T}{\log \log T}}
  \Big),
\end{equation*}
for a constant $c\geq \frac{1}{\sqrt{2}}$,
then they made a breakthrough\cite{BS18} to reach $c = 1$, and La Bret\`eche and Tenenbaum\cite{dT19} improved it to $c=\sqrt{2}$.
Their works are based on improved estimates of the GCD sum.
Recently, Darbar and Maiti\cite{DM24,DM25} extended the GCD method to quadratic Dirichlet \( L \)-functions, under suitable conditions, they were able to obtain results involving a triple-logarithmic term.
\end{remark}

\begin{remark}
Under the Riemann Hypothesis (RH) and the Generalized Riemann Hypothesis (GRH) for Dirichlet \( L \)-functions associated with \( \chi_p \), results of the form
\[
\exp\left(c \frac{\sqrt{\log X}}{\sqrt{\log \log X}}\right)
\]
can be readily obtained using Soundararajan’s method.
Based on the expected square-root cancellation in the first moment error term, comparison with Soundararajan’s paper suggests that one should have \( c = 2 \times \frac{1}{\sqrt{4}} = 1 \).
Furthermore, as indicated in \cite{DM24,DM25},
it might be possible to obtain an even stronger bound of the form
\[
\exp\left(c \frac{\sqrt{\log X \log \log \log X}}{\sqrt{\log \log X}}\right),
\]
which would be an interesting problem to explore.
\end{remark}

\begin{remark}
  Recently, Heap and Li\cite{HL23} established simultaneous extreme values results for lots of cases, included Riemann $\zeta$-function, primitive Dirichlet $L$-functions, $L$-functions attached to holomorphic or Maass cusp forms over $\SL_2(\mathbb{Z})$,
  $L$-functions attached to some fixed primitive cusp forms with primitive characters twists or quadratic twists, and $L$-functions attached to irreducible unitary cuspidal automorphic representations.
\end{remark}

The rest of this paper is arranged as follows.
In \S \ref{sec:preliminaries}, we introduce some notations and present some lemmas that we will apply later.
In \S \ref{sec:proofoftheorem}, we briefly describe the proof of Theorem \ref{thm1} based on Soundararajan's resonance method.
In \S \ref{proofofprop1}, we show the contribution from the diagonal terms.
In \S \ref{proofofprop2}, we show the contribution from the off-diagonal terms, due to Baluyot--Pratt's strategy, by employing three different methods to handle summations over different intervals.
More precisely, using Page's theorem to compute the contribution from at most one Siegel zero, using zero-density estimate and using Vaughan's identity.
In \S \ref{remarks}, we will list some
results about the distribution of central values of automorphic $L$-functions in quadratic twists families.

\section{Notation and preliminary results}\label{sec:preliminaries}

For prime $p\equiv 1\pmod 8$, the Dirichlet $L$-function with respect to $\chi_p(\cdot)$ is
\begin{equation*}\label{def:dirichletlfun}
  L(s,\chi_p)=\sum_{n=1}^{\infty}\chi_p(n)n^{-s}
\end{equation*}
when $\Re s>1$.

\begin{lemma}[\cite{IK04}, Theorem 5.3] \label{lemma:AFE}
For prime $p\equiv 1\pmod 8$, we have
\begin{equation*}
  L(\frac{1}{2},\chi_{p})
  =\frac{2}{(1-\frac{1}{\sqrt{2}})^{2}}
  \sum_{\substack{n=1 \\ n \text{ odd} }}^\infty
  \frac{\chi_p(n)}{\sqrt{n}}
  V\bigg(n\sqrt{\frac{\pi}{p}}\bigg),
 \end{equation*}
where $V(y)$ is defined by
\begin{equation*} \label{eqn:V}
V(y):=\frac{1}{2\pi i}\int_{(u)}
\frac{\Gamma(\frac{s}{2}+\frac{1}{4})}
{\Gamma(\frac{1}{4})}
\bigg(1-\frac{1}{2^{\frac{1}{2}-s}}\bigg)
y^{-s} e^{s^2} \frac{\dd s}{s}
\end{equation*}
with $u>0$.  Here, and in sequel, $\int_{(u)}$ stands for $\int_{u-i \infty}^{u+i \infty}$.
The function $V(y)$ is smooth in $y>0$ and satisfies $y^hV^{(h)}(y)\ll_{A,N} y^{-A}$ for non-negative integers $h$ and large positive constant $A$, and we have $V(y)=1-\frac{1}{\sqrt{2}}+O(y^{\frac{1}{2}-\varepsilon})$ as $y\to 0$.
\end{lemma}

Let $\Phi$ be a smooth and nonnegative function in Schwarz class supported on $[1,2]$ with $\Phi(x)=1$ for $x\in [\frac{7}{6},\frac{11}{6}]$.
For any complex number $s$, we define
\begin{equation}\label{def:checkPhi}
  \Phi_s(x):=\Phi(x)x^\frac{s}{2}\textrm{, and }  \check\Phi(s):=\int_{0}^{\infty}\Phi(x)x^s\dd x.
\end{equation}

\begin{lemma}[\cite{RS15}, Proposition 1]\label{lemma:poission}
Writing $\sum\nolimits^\flat$ as the summation over square-free numbers.
For any $\varepsilon>0$ and
$n< X^{1-\varepsilon}$ such that $(n,8)=1$,
we have
  \begin{equation*}
   \mathop{\sum\nolimits^\flat}\limits_{d\equiv 1\pmod 8} \chi_d(n)\Phi\bigg(\frac{d}{X}\bigg)=
    \check{\Phi}(0)\frac{X}{\pi^2}\prod_{p\mid n}
    \bigg(1+\frac{1}{p}\bigg)^{-1}\delta_{n=\square}+
    O(X^{\frac{1}{2}+\varepsilon}\sqrt{n}),
  \end{equation*}
    where the symbol $\delta_{n=\square}$ equals to $1$ if $n$ is a square and 0 otherwise.
\end{lemma}

\section{Soundararajan's resonance method}\label{sec:proofoftheorem}

\begin{proof}[Proof of Theorem \ref{thm1}]
We introduce $M=X^{\theta}$ with some $\theta< 1$.
Set $\mathcal{L}=\sqrt{\log M\log\log M}$.
Following Soundararajan's resonance method\cite{So08}, we choose $b(\cdot)$ to be a multiplicative function supported on square-free numbers with $b(p)=\frac{\mathcal{L}}{\sqrt{p}\log p}$ for $\mathcal{L}^2\leq p \leq \exp((\log \mathcal{L})^2)$ and $b(p)=0$ otherwise.
Let

\begin{equation*}\label{def:resonator}
  R(d)=\sum_{\substack{m\leq M \\ m \text{ odd}}}b(m)\chi_{d}(m).
\end{equation*}

In the following, we will omit the odd condition of the variables due to the definition of $b(\cdot)$.
Let

\begin{align*}
  M_1(R,X)
  &=\mathop{\sum\nolimits^\flat}\limits_{d\equiv 1\pmod 8}
  R(d)^2\Phi\bigg(\frac{d}{X}\bigg) \\
  &=\underset{\substack{m_1,m_2\leq M}}{\sum\sum}b(m_1)b(m_2)
 \mathop{\sum\nolimits^\flat}\limits_{d\equiv 1\pmod 8}
 \chi_d(m_1m_2)
  \Phi\bigg(\frac{d}{X}\bigg)
  =\mathcal{M}_1+\mathcal{R}_1,
\end{align*}
where, using Lemma\ref{lemma:poission}, $m\le M<\frac{X}{2}\le p$ and trivial estimates, and recall that the function $b(\cdot)$ is supported on square-free numbers, we have
\begin{equation*}
  \mathcal{M}_1=\sum_{m\leq M} b(m)^2 \mathop{\sum\nolimits^\flat}\limits_{d\equiv 1\pmod 8}\chi_{d}(m^2)
  \Phi\bigg(\frac{d}{X}\bigg)
  =  \check{\Phi}(0)\frac{X}{\pi^2}\sum_{m\leq M} b(m)^2\prod_{p\mid m}
    \bigg(1+\frac{1}{p}\bigg)^{-1} +
    O(X^{\frac{1}{2}+\varepsilon}M^2)
\end{equation*}
is the diaogonal part of it and
\begin{equation*}
  \mathcal{R}_1
  =\underset{\substack{m_1,m_2\leq M\\m_1\neq m_2}}{\sum\sum}b(m_1)b(m_2)
  \mathop{\sum\nolimits^\flat}\limits_{d\equiv 1\pmod 8}
 \chi_d(m_1m_2)\Phi\bigg(\frac{d}{X}\bigg)
 =O(X^{\frac{1}{2}+\varepsilon}M^3)
\end{equation*}
is the off-diaogonal part.
In Proposition \ref{prop:Mi} below, we estimate the size of $M_1$, so here we let $\theta < \frac{1}{6}$.

Let $p$ be denoted as a prime number,
and
\begin{align*}
  M_2(R,X)
  &=\sum_{p\equiv 1\pmod 8 }
  (\log p) L\bigg(\frac{1}{2},\chi_p\bigg)R(p)^2 \Phi\bigg(\frac{p}{X}\bigg)\\
  &=\frac{2}{(1-\frac{1}{\sqrt{2}})^{2}}
  \sum_{\substack{p\equiv 1\pmod 8 }}
  (\log p)\Phi\bigg(\frac{p}{X}\bigg)
  \sum_{\substack{n=1 \\ n \textrm{ odd }}}^\infty
  \underset{\substack{m_1,m_2\leq M}}{\sum\sum}b(m_1)b(m_2)
  \frac{\chi_p(nm_1m_2)}{\sqrt{n}}
  V\bigg(n\sqrt{\frac{\pi}{p}}\bigg) \\
  &=:\mathcal{M}_2+\mathcal{R}_2,
\end{align*}
where,
after removing the coprime restrictions for $m_1,m_2$ with respect to $p$ since $m_1,m_2\le M<\frac{X}{2}\le p$,

\begin{equation*}
\begin{split}
  &\mathcal{M}_2=\frac{2}{(1-\frac{1}{\sqrt{2}})^{2} }
  \sum_{p\equiv 1\pmod 8 }
  (\log p)\Phi\bigg(\frac{p}{X}\bigg)
  \frac{1}{2\pi i}
  \int_{(u)}
\frac{\Gamma(\frac{s}{2}+\frac{1}{4})}
{\Gamma(\frac{1}{4})}
\bigg(1-\frac{1}{2^{\frac{1}{2}-s}}\bigg)\\
&\hskip 30pt\times
  \underset{\substack{nm_1m_2=\square\\
  m_1,m_2\leq M\\(n,2p)=1}}{\sum\sum\sum}
\frac{b(m_1)b(m_2)}
{n^{\frac{1}{2}+s}}
\bigg(\frac{p}{\pi}\bigg)^{\frac{s}{2}}
e^{s^2}\frac{\dd s}{s}
\end{split}
\end{equation*}
and

\begin{equation*}
\begin{split}
  &\mathcal{R}_2=\frac{2}{(1-\frac{1}{\sqrt{2}})^{2} }
  \sum_{p\equiv 1\pmod 8 }
  (\log p)\Phi\bigg(\frac{p}{X}\bigg)
  \frac{1}{2\pi i}
  \int_{(u)}
\frac{\Gamma(\frac{s}{2}+\frac{1}{4})}
{\Gamma(\frac{1}{4})}
\bigg(1-\frac{1}{2^{\frac{1}{2}-s}}\bigg)\\
&\hskip 30pt\times
  \underset{\substack{nm_1m_2\neq \square\\
  m_1,m_2\leq M\\(n,2)=1}}{\sum\sum\sum}
\frac{b(m_1)b(m_2)}
{n^{\frac{1}{2}+s}}\chi_p(nm_1m_2)
\bigg(\frac{p}{\pi}\bigg)^{\frac{s}{2}}
e^{s^2}\frac{\dd s}{s}.
\end{split}
\end{equation*}

We have the following two propositions which will be proved in later sections.
\begin{proposition}\label{prop:Mi}
For $\theta< 1/6$,
  \begin{equation*}
  \mathcal{M}_1\ll X\prod_{\mathcal{L}^2\leq p \leq \exp((\log \mathcal{L})^2)}
(1+ b(p)^2)
\end{equation*}
and
\begin{equation*}
   |\mathcal{M}_2|\gg
   X\log X \prod_{\mathcal{L}^2\leq p \leq \exp((\log \mathcal{L})^2)}
  \bigg(1+2\frac{b(p)}{\sqrt{p}}+b(p)^2\bigg).
\end{equation*}
\end{proposition}

\begin{proposition}\label{prop:R2}
For $\theta\leq \frac{2}{45}-3\varepsilon$,
\begin{equation*}
  \mathcal{R}_2=O(\mathcal{M}_2X^{-\varepsilon}).
  \end{equation*}
\end{proposition}

Choose $\theta= \frac{2}{45}-3\varepsilon$,
then we end the proof of Theorem \ref{thm1} by using
\begin{multline*}
  \max_{\substack{X< p\leq 2X\\p\equiv 1\pmod 8}}
  L(\frac{1}{2},\chi_p)\geq
\frac{M_2(R,X)}{ M_1(R,X)\log 2X}
\gg
\frac{\mathcal{M}_1}{|\mathcal{M}_2|\log X}
  \gg \prod_{\mathcal{L}^2\leq p \leq \exp((\log \mathcal{L})^2)}
  \bigg(1+2\frac{b(p)}{\sqrt{p}}\bigg)
 \\ \gg \exp\bigg(2\mathcal{L} \sum_{\mathcal{L}^2\leq p \leq \exp((\log \mathcal{L})^2)} \frac{1}{p\log p}\bigg)
 =\exp\Big((2\sqrt{\theta}+o(1))\frac{\sqrt{\log X}}{\sqrt{\log\log X}}\Big).
\end{multline*}
\end{proof}

\section{Estimation of diagonal terms}\label{proofofprop1}
\begin{proof}[Proof of Proposition \ref{prop:Mi}]
For $m_1,m_2$ square-free, we write $m_1m_2=m_0(m_1,m_2)^2$ so that $m_0$ is square-free.
Thus
\begin{equation*}
\begin{split}
  \mathcal{M}_2&=
   \frac{2}{(1-\frac{1}{\sqrt{2}})^{2}}
  \underset{m_1,m_2\leq M}{\sum\sum}
  b(m_1)b(m_2)
  \sum_{p\equiv 1\pmod 8 }
  (\log p)\Phi\bigg(\frac{p}{X}\bigg)
  \\
&\hskip 30pt\times
  \frac{1}{2\pi i}
  \int_{(u)}
\frac{\Gamma(\frac{s}{2}+\frac{1}{4})}
{\Gamma(\frac{1}{4})}
(1-2^{-\frac{1}{2}+s})
  \sum_{\substack{nm_0=\square\\(n,2p)=1}}
\frac{1}
{n^{\frac{1}{2}+s}}
\bigg(\frac{p}{\pi}\bigg)^{\frac{s}{2}}
e^{s^2}\frac{\dd s}{s}\\
&=
 \frac{2}{(1-\frac{1}{\sqrt{2}})^{2}}
 \underset{m_1,m_2\leq M}{\sum\sum}
  \frac{b(m_1)b(m_2)}{\sqrt{m_0}}
  \sum_{p\equiv 1\pmod 8 }
  (\log p)\Phi\bigg(\frac{p}{X}\bigg)
  \\
&\hskip 30pt\times
  \frac{1}{2\pi i}
  \int_{(u)}
\frac{\Gamma(\frac{s}{2}+\frac{1}{4})}
{\Gamma(\frac{1}{4})}
\zeta(1+2s)\bigg(\frac{p}{\pi m_0^2}\bigg)^{\frac{s}{2}}
(1-2^{-\frac{1}{2}+s})
(1-2^{-1-2s})
(1-p^{-1-2s})
e^{s^2}\frac{\dd s}{s},
\end{split}
\end{equation*}

Let $s=u+it$, and
\begin{equation*}
  \mathcal{I}_{p,m_0}=\frac{1}{2\pi i}
  \int_{(u)}
\frac{\Gamma(\frac{s}{2}+\frac{1}{4})}
{\Gamma(\frac{1}{4})}
\zeta(1+2s)
\bigg(\frac{p}{\pi m_0^2}\bigg)^{\frac{s}{2}}
(1-2^{-\frac{1}{2}+s})
(1-2^{-1-2s})
(1-p^{-1-2s})
e^{s^2}\frac{\dd s}{s},
\end{equation*}
we move the line of integration to $\Re(s)=-\frac{1}{2}+\varepsilon$,
leaving a residue at $s=0$,
the new integral is
\begin{equation*}
\begin{split}
  &\frac{1}{2 \pi^{3/4 + \varepsilon/2}}
  p^{-\frac{1}{4}+\frac{\varepsilon}{2}} m_0^{\frac{1}{2}-\varepsilon}\\
&\hskip 20pt\times
\int_{-\infty}^{+\infty}
\frac{\Gamma(\frac{\varepsilon+it}{2})}
{\Gamma(\frac{1}{4})}
\zeta(2\varepsilon+2it)
\bigg(\frac{p}{\pi m_0^2}\bigg)^{\frac{it}{2}}
(1-2^{-1+\varepsilon+it})
(1-2^{-2\varepsilon-2it})
(1-p^{-2\varepsilon-2it})
\frac{e^{(-\frac{1}{2}+\varepsilon+it)^2}}
{-\frac{1}{2}+\varepsilon+it}\dd t\\
&\ll p^{-\frac{1}{4}+\varepsilon}
m_0^{\frac{1}{2}}
\int_{-\infty}^{+\infty}
|\zeta(2\varepsilon+2it)|
\frac{e^{-t^2}}{1+|t|} \dd t\\
&\ll p^{-\frac{1}{4}+\varepsilon}
m_0^{\frac{1}{2}}.
\end{split}
\end{equation*}
We have the Laurent series expansions in a neighborhood of $s=0$
\begin{equation*}
\begin{gathered}
\frac{\Gamma(\frac{s}{2}+\frac{1}{4})}
{\Gamma(\frac{1}{4})}
=1+\frac{1}{2}\frac{\Gamma^{'}(\frac{1}{4})}
{\Gamma(\frac{1}{4})}s+\dots
=1-\bigg(\frac{\pi}{4}+\frac{3\log 2}{2} +\frac{\gamma}{2}\bigg)s+\dots;\\
\zeta(1+2s)=\frac{1}{2s}+\gamma+\dots;\\
\bigg(\frac{p}{\pi m_0^2}\bigg)
^{\frac{s}{2}}=1+\frac{1}{2}
\log\bigg(\frac{p}{\pi m_0^2}\bigg)s+\dots;\\
(1-2^{-\frac{1}{2}+s})
(1-2^{-1-2s})
(1-p^{-1-2s})
e^{s^2}
=2^{-1}(1-2^{-\frac{1}{2}})
(1-p^{-1})
\\ \times
\bigg(1+\bigg(2+2\log 2+\frac{2\log p}{p-1}-\frac{\log 2}{\sqrt{2}-1}\bigg)s+\dots\bigg);
\end{gathered}
\end{equation*}
where $\gamma$ is the Euler constant.
It follows that the residues can be written as
\begin{equation}\label{eq:residue}
  2^{-1}(1-2^{-\frac{1}{2}})
(1-p^{-1})\Big(\frac{1}{4}\log\bigg(\frac{p}{m_0^2}\bigg)
+\frac{\log p}{p-1}+ c_0
\Big)
\end{equation}
where fixed constant
\begin{equation*}
  c_0=-\bigg(\frac{\pi}{4}+\frac{3\log 2}{2} +\frac{\gamma}{2}\bigg)
-\frac{\log(\pi)}{4}
+\frac{3}{2}-\frac{\log 2}{2\sqrt{2}-2}+\gamma,
\end{equation*}
and $c_0\approx -1.15939623$ is easily known by hand or by MATLAB.
Then for $X$ large enough, $p\geq X$, we have
\begin{equation*}
  |\mathcal{I}_{p,m_0}|\geq
  2^{-1}(1-2^{-\frac{1}{2}})
(1-X^{-1})\bigg(\frac{1}{4}
\log\bigg(\frac{p}{m_0^2}\bigg)
-1.15939624\bigg)
\gg \log X
\end{equation*}
since $m_0\le M^2=X^{2\theta}$.

Therefore, we can get
\begin{equation}\label{M2 gg}
\begin{split}
  |\mathcal{M}_2|&\gg
 \underset{m_1,m_2\leq M}{\sum\sum}
  \frac{b(m_1)b(m_2)}{\sqrt{m_0}}
  \sum_{p\equiv 1\pmod 8 }
  (\log p)\Phi\bigg(\frac{p}{X}\bigg)\log X\\
  &\gg_\Phi X\log X
  \underset{m_1,m_2\leq M}{\sum\sum}
  \frac{b(m_1)b(m_2)}{\sqrt{m_0}},
  \end{split}
\end{equation}
by applying the prime number theorem in arithmetic progressions
\begin{equation*}
  \sum_{\substack{p\equiv 1\pmod 8 \\ p \textrm{ prime }}}
  (\log p)\Phi\bigg(\frac{p}{X}\bigg)
  =\frac{X}{4}\check{\Phi}(0)+O\bigg(\frac{X}{\log X}\bigg),
\end{equation*}
where $m_0=\frac{m_1m_2}{(m_1, m_2)^2}$.

Both $m_1,m_2$ in \eqref{M2 gg} are square-free, then we write $m_1=qr$, $m_2=qs$, $m_0=rs$ with $(q,r)=(q,s)=(r,s)=1$, we have
\begin{equation}\label{truncated sum}
    \underset{m_1,m_2\leq M}{\sum\sum}
  \frac{b(m_1)b(m_2)}{\sqrt{m_0}}
  =\underset{\substack{(q,r)=(q,s)=(r,s)=1
  \\qr,qs\leq M}}
  {\sum\sum\sum}
  \frac{b(qr)b(qs)}{\sqrt{rs}}.
\end{equation}
We remove the restriction of $qr,qs\leq M$, the extend sum is
\begin{equation}\label{complete sum}
\underset{\substack{(q,r)=(q,s)=(r,s)=1}}
  {\sum\sum\sum}
  \frac{b(qr)b(qs)}{\sqrt{rs}}
    =\prod_{\mathcal{L}^2\leq p \leq \exp((\log \mathcal{L})^2)}
  \bigg(1+2\frac{b(p)}{\sqrt{p}}+b(p)^2\bigg).
\end{equation}

By the symmetry of $qr$ and $qs$, the difference between \eqref{truncated sum} and \eqref{complete sum} could be controlled by
  \begin{equation*}
2\underset{\substack{(q,r)=(q,s)=(r,s)=1
  \\qr> M}}
  {\sum\sum\sum}
  \frac{b(q)^2b(r)b(s)}{\sqrt{rs}}.
\end{equation*}
By employing Rankin's trick, for any $\alpha>0$, the above is
\begin{equation}\label{eqn:Rankinerror}
\begin{split}
 &\ll M^{-\alpha}\underset{
 \substack{(q,r)=(q,s)=(r,s)=1}}
  {\sum\sum\sum}b(q)^2 b(r)b(s)
  q^{\alpha}r^{-\frac{1}{2}+\alpha}
s^{-\frac{1}{2}}\\
&\ll M^{-\alpha}
\prod_{\mathcal{L}^2\leq p \leq \exp((\log \mathcal{L})^2)}
\Big(1+b(p)^2p^{\alpha}
  +b(p)p^{-\frac{1}{2}}(1+p^{\alpha})\Big),
\end{split}
\end{equation}
Choose $\alpha=(\log \mathcal{L})^{-3}$, we see the ratio of the upper bound in \eqref{eqn:Rankinerror} to the right hand side in \eqref{complete sum} is
\begin{equation*}\label{eqn:ratioerror}
\begin{split}
&\ll \exp\Bigg(
  -\alpha\log M+\sum_{\mathcal{L}^2\leq p\leq \exp((\log \mathcal{L})^2)}\Big(
 b(p)^2(p^{\alpha}-1)
 +b(p)p^{-\frac{1}{2}}(p^{\alpha}-1)
  \Big)
  \Bigg)\\
  &\ll \exp\Bigg(
  -\alpha\log M+\sum_{\mathcal{L}^2\leq p\leq \exp((\log \mathcal{L})^2)}\bigg(
  \frac{\mathcal{L}}{p\log p}(\frac{\mathcal{L}}{\log p}+1)(p^\alpha-1)
  \bigg)
  \Bigg)\\
     & \ll \exp(-\alpha\frac{\log M}{\log\log M}).
\end{split}
\end{equation*}
Thus
\begin{equation*}
   |\mathcal{M}_2|\gg
   X\log X \prod_{\mathcal{L}^2\leq p \leq \exp((\log \mathcal{L})^2)}
  \bigg(1+2\frac{b(p)}{\sqrt{p}}+b(p)^2\bigg).
\end{equation*}

Similarly we have
\begin{equation*}
  \mathcal{M}_1\ll X\prod_{\mathcal{L}^2\leq p \leq \exp((\log \mathcal{L})^2)}
\bigg(1+\frac{p b(p)^2}{p+1}\bigg)\ll X\prod_{\mathcal{L}^2\leq p \leq \exp((\log \mathcal{L})^2)}
(1+ b(p)^2).
\end{equation*}
\end{proof}

\section{Estimation of off-diagonal terms}
\label{proofofprop2}
\begin{proof}[Proof of Proposition \ref{prop:R2}]
Recall that
\begin{align*}
    \mathcal{R}_2=\frac{2}{\left( 1-\frac{1}{\sqrt{2}}\right)^2}\mathop{\sum_{m_1\leq M} \sum_{m_2\leq M} \sum_{\substack{\textrm{odd }n=1 }}^{\infty}}_{m_1m_2n \neq \square} \frac{b(m_1)b(m_2)}{\sqrt{n}}   \sum_{p\equiv 1 \, (\text{mod }8) } (\log p) \Phi\left(\frac{p}{X}\right)V\left( n\sqrt{\frac{\pi}{p}}\right) \left(\frac{p}{m_1m_2n}\right).
\end{align*}
We can uniquely write $n = rk^2$, where $r$ is square-free and $k$ is an integer. It follows that
\begin{align*}
    \mathcal{R}_2=\frac{2}{\left( 1-\frac{1}{\sqrt{2}}\right)^2} &\sum_{m_1\leq M} \sum_{m_2\leq M} b(m_1)b(m_2) \sum_{\substack{\textrm{odd } r=1 \\ m_1m_2r \neq \square}}^{\infty} \sum_{\substack{k=1 \\ k \text{ odd}}}^{\infty} \frac{{\mu (r)}^2}{k \sqrt{r}} \\
    &\times\sum_{p\equiv 1 \, (\text{mod }8) } (\log p) \Phi\left(\frac{p}{X}\right)V\left( rk^2\sqrt{\frac{\pi}{p}}\right) \left(\frac{p}{m_1m_2rk^2}\right).
\end{align*}
We next factor out the greatest common divisor, say $q$, of $m_1$ and $m_2$, and change variables $m_1 \rightarrow qm_1, m_2 \rightarrow qm_2$ to obtain
\begin{align*}
    \mathcal{R}_2=\frac{2}{\left( 1-\frac{1}{\sqrt{2}}\right)^2} &\sum_{q\leq M} b(q)^2 \sum_{\substack{m_0\leq \frac{M^2}{q^2}\\ (m_0,q)=1}} a_q(m_0)b(m_0) \sum_{\substack{\textrm{odd } r=1 \\ m_0r \neq \square }}^{\infty} \sum_{\substack{k=1 \\ k \text{ odd}}}^{\infty} \frac{{\mu (r)}^2}{k \sqrt{r}} \\
    &\times\sum_{p\equiv 1 \, (\text{mod }8) } (\log p) \Phi\left(\frac{p}{X}\right)V\left( rk^2\sqrt{\frac{\pi}{p}}\right) \left(\frac{p}{m_0rq^2k^2}\right),
\end{align*}
where
\begin{align*}
a_q(m_0) = \#\{m_1 | m_0 : (m_1, m_0 / m_1) = 1,~\textrm{and}~m_1, m_0/m_1 \leq M/q~(\textrm{odd, square-free})\}\leq d(m_0).
\end{align*}
Similarly, we factor out the greatest common divisor, say $g$, of $m_0$ and $r$, and change variables $m_0 \rightarrow gm_0, r \rightarrow gr$, and from $m_0r\neq \square$ we have $m_0r>1$, thus
\begin{align*}
    \mathcal{R}_2=\frac{2}{\left( 1-\frac{1}{\sqrt{2}}\right)^2} &\sum_{q\leq M} b(q)^2 \sum_{\substack{g\leq \frac{M^2}{q^2}\\ (g,q)=1}} \frac{1}{\sqrt{g}} \sum_{\substack{m_0 \leq \frac{M^2}{gq^2} \\ (m_0,qg)=1}} a_q(m_0g)b(m_0g) \sum_{\substack{\textrm{odd } r=1 \\ (r,gm_0)=1 \\ m_0r>1}}^{\infty} \frac{{\mu (r)}^2}{ \sqrt{r}} \sum_{\substack{k=1 \\ k \text{ odd}}}^{\infty} \frac{1}{k} \\
    &\times\sum_{p\equiv 1 \, (\text{mod }8) } (\log p) \Phi\left(\frac{p}{X}\right)V\left( grk^2\sqrt{\frac{\pi}{p}}\right) \left(\frac{p}{m_0rg^2q^2k^2}\right),
\end{align*}
Then we will drop $\left(\frac{p}{g^2q^2k^2}\right)$.
Conditions $ p \nmid gq$ are automatically satisfied since $q \leq M < \frac{X}{2} \leq p $ and $g \leq M^2 < \frac{X}{2} \leq p $.
By using the decay property of $V(\cdot)$ in Lemma \ref{lemma:AFE}, and recall $|d(\cdot)b(\cdot)|\leq 1$, we truncate the sum over $k$ to $k \leq X^{\frac{1}{4}+\varepsilon}$ and the sum over $r$ to $r \leq X^{\frac{1}{2}+\varepsilon}$ at the cost of an error term
\begin{align*}
&X^{-3}\sum_{q\leq M} b(q)^2 \sum_{\substack{g\leq \frac{M^2}{q^2}\\ (g,q)=1}} \frac{b(g)}{ \sqrt{g}} \sum_{\substack{m_0 \leq \frac{M^2}{gq^2} \\ (m_0,qg)=1}} a_q(m_0g) b(m_0)\\
&\ll X^{-3}\sum_{q\leq M} b(q)^2 \sum_{\substack{g\leq \frac{M^2}{q^2}\\ (g,q)=1}} \frac{d(g)b(g)}{ \sqrt{g}} \sum_{m_0 \leq M^2} d(m_0)b(m_0)\\
&\ll X^{-3}M^{2}\prod_{\mathcal{L}^2\leq p \leq \exp((\log \mathcal{L})^2)}\bigg(1+2\frac{b(p)}{\sqrt{p}}+b(p)^2\bigg).
\end{align*}
Moreover, by quadratic reciprocity, we can get
\begin{align*}
    \mathcal{R}_2=\frac{2}{\left( 1-\frac{1}{\sqrt{2}}\right)^2} &\sum_{q\leq M} b(q)^2 \sum_{\substack{g\leq \frac{M^2}{q^2}\\ (g,q)=1}} \frac{b(g)}{ \sqrt{g}} \sum_{\substack{m_0 \leq \frac{M^2}{gq^2} \\ (m_0,qg)=1}} a_q(m_0g)b(m_0) \sum_{\substack{r \leq X^{\frac{1}{2}+\varepsilon } \\ (r,2gm_0)=1 \\ m_0r>1}} \frac{{\mu (r)}^2}{ \sqrt{r}} \sum_{\substack{k=1 \\ k \text{ odd}}}^{\infty} \frac{1}{k} \\
    &\times\sum_{p\equiv 1 \, (\text{mod }8) } (\log p) \Phi\left(\frac{p}{X}\right)V\left( grk^2\sqrt{\frac{\pi}{p}}\right) \left(\frac{m_0r}{p}\right) \\
    &+ O\bigg(X^{-1}\prod_{\mathcal{L}^2\leq p \leq \exp((\log \mathcal{L})^2)}\bigg(1+2\frac{b(p)}{\sqrt{p}}+b(p)^2\bigg)\bigg).
\end{align*}
By orthogonality for characters modulo 8, elementray arguments for moduli of characters(see\cite[Theorem 2.2.15]{Co07}) and inserting the definition of $V(\cdot)$, the $k$-summation turns into a zeta function, and we have
\begin{equation}\label{eq: pretreatment of R2 1}
  \begin{split}
  \mathcal{R}_2=& \frac{2}{\left( 1-\frac{1}{\sqrt{2}}\right)^2} \sum_{q\leq M} b(q)^2 \sum_{\substack{g\leq \frac{M^2}{q^2}\\ (g,q)=1}} \frac{{\mu (g)}^2b(g)}{ \sqrt{g}} \sum_{\substack{m_0 \leq \frac{M^2}{gq^2} \\ (m_0,qg)=1}} a_q(m_0g) b(m_0)\sum_{\substack{r \leq X^{\frac{1}{2}+\varepsilon } \\ r \text{ odd} \\ (r,gm_0)=1 \\ m_0r>1}}\frac{{\mu (r)}^2}{ \sqrt{r}}   \\
  &\times \frac{1}{4} \sum_{\gamma \in \{\pm 1, \pm 2\}} \frac{1}{2\pi i} \int_{(c)} \frac{\Gamma(\frac{s}{2}+\frac{1}{4})}{\Gamma(\frac{1}{4})} \left( 1-\frac{1}{2^{\frac{1}{2}-s}}\right)\left(1 - \frac{1}{2^{1+2s}}\right)\zeta(1+2s)\pi^{-s/2} \left( gr\right)^{-s} \\
  &\times \sum_{p } (\log p) \Phi\left(\frac{p}{X}\right)  \left(\frac{\gamma m_0r}{p}\right) p^{s/2} \, e^{s^2}\frac{\dd s}{s}+O\bigg(\prod_{\mathcal{L}^2\leq p \leq \exp((\log \mathcal{L})^2)}\bigg(1+2\frac{b(p)}
  {\sqrt{p}}+b(p)^2\bigg)\bigg) .
  \end{split}
\end{equation}
Let $\chi_{\gamma m_0r}(\cdot)=\left(\frac{\gamma m_0r}{\cdot }\right)$ if $\gamma m_0r \equiv 1$ (mod $4$), and $\chi_{\gamma m_0r}(\cdot)=\left(\frac{4\gamma m_0r}{\cdot }\right)$ if $\gamma m_0r\equiv 2$ or $3$ (mod $4$).
Then each $\chi_{\gamma m_0r}(\cdot)$ is a real primitive character for all the relevants $\gamma,m_0,r$.
Since $m_0r>1$, we see that $\gamma m_0r$ is never $1$, so each $\chi_{\gamma mr}$ is nonprincipal.
We choose $c=\frac{1}{\log X}$ so that $p^{\frac{s}{2}}$ is bounded absolutely.

It is more convenient to replace the $\log p$ factor with the von Mangoldt function $\Lambda(n)$. By trivial estimation we have
\begin{equation}\label{eq:extend to whole summation}
\sum_p(\log p) \Phi\left(\frac{p}{X}\right) \chi_{\gamma m_0 r}(p) p^{s / 2}=\sum_n \Lambda(n) \Phi\left(\frac{n}{X}\right) \chi_{\gamma m_0 r}(n) n^{s / 2}+O\left(X^{1 / 2}\right),
\end{equation}
since $p^{s/2}$ is bounded on $\Re(s) = 1 / \log X$.
Inserting this into \eqref{eq: pretreatment of R2 1}, we find the error term in \eqref{eq:extend to whole summation} costs
\begin{align*}
&X^{\frac{3}{4}+\varepsilon}\sum_{q\leq M} b(q)^2 \sum_{\substack{g\leq \frac{M^2}{q^2}\\ (g,q)=1}} \frac{b(g)}{ \sqrt{g}} \sum_{\substack{m_0 \leq \frac{M^2}{gq^2} \\ (m_0,qg)=1}} a_q(m_0g) b(m_0)\\
\ll&X^{\frac{3}{4}+\varepsilon}\sum_{q\leq M} b(q)^2 \sum_{\substack{g\leq \frac{M^2}{q^2}\\ (g,q)=1}} \frac{d(g)b(g)}{ \sqrt{g}} \sum_{m_0 \leq M^2} d(m_0)b(m_0)\\
\ll&X^{\frac{3}{4}+\varepsilon}M^2\prod_{\mathcal{L}^2\leq p \leq \exp((\log \mathcal{L})^2)}\bigg(1+2\frac{b(p)}{\sqrt{p}}+b(p)^2\bigg),
\end{align*}
which is acceptable when $M\ll X^{1/8-\varepsilon}$.
By the rapid decay of $|\Gamma(\frac{s}{2}+\frac{1}{4})\zeta(1+2s)
e^{s^2}|$ in vertical strips, and
$$|\sum_n \Lambda(n) \Phi\left(\frac{n}{X}\right) \chi_{\gamma m_0 r}(n) n^{s / 2}|
\ll X,$$
we can truncate the integral to $|\Im(s)| \leq (\log X)^2$, at the cost of a negligible error term $O(X^{-2023})$.
Thus
\begin{equation}\label{eq: pretreatment of R2 2}
  \begin{split}
  \mathcal{R}_2=
  &\frac{1}{2\left( 1-\frac{1}{\sqrt{2}}\right)^2} \sum_{q\leq M} b(q)^2 \sum_{\substack{g\leq \frac{M^2}{q^2}\\ (g,q)=1}} \frac{{\mu (g)}^2b(g)}{ \sqrt{g}} \sum_{\substack{m_0 \leq \frac{M^2}{gq^2} \\ (m_0,qg)=1}} a_q(m_0g)b(m_0) \sum_{\substack{r \leq X^{\frac{1}{2}+\varepsilon } \\ (r,2gm_0)=1 \\ m_0r>1}}
   \frac{{\mu (r)}^2}{ \sqrt{r}}   \\
  &\times \sum_{\gamma \in \{\pm 1, \pm 2\}} \frac{1}{2\pi i} \int_{\frac{1}{\log X}-i(\log X)^2}^{\frac{1}{\log X}+i(\log X)^2} \frac{\Gamma(\frac{s}{2}+\frac{1}{4})}{\Gamma(\frac{1}{4})} \left( 1-\frac{1}{2^{\frac{1}{2}-s}}\right)\left(1 - \frac{1}{2^{1+2s}}\right)\zeta(1+2s)\left(\frac{X}{\pi}\right)^{\frac{s}{2}}e^{s^2} \left( gr\right)^{-s} \\
  &\times \sum_{n} \Lambda (n) \Phi_s\left(\frac{n}{X}\right)  \chi_{\gamma m_0 r}(n) \frac{\dd s}{s}
  +O\bigg(X^{1-\varepsilon }\prod_{\mathcal{L}^2\leq p \leq \exp((\log \mathcal{L})^2)}\bigg(1+2\frac{b(p)}
  {\sqrt{p}}+b(p)^2\bigg)\bigg),
  \end{split}
\end{equation}
where $\Phi_s$ is defined by \eqref{def:checkPhi}.

Up to now, we have accomplished the pretreatment of $\mathcal{R}_2$.
We keep on showing $\mathcal{R}_2$ is small enough by dividing the interval into three segments and call these ranges Regimes \uppercase\expandafter{\romannumeral1}, \uppercase\expandafter{\romannumeral2} and \uppercase\expandafter{\romannumeral3}, which correspond to small, medium, and large values of $m_0r$, consistent with that in \cite{BP22}.
Regime \uppercase\expandafter{\romannumeral1} corresponds to
\begin{align*}
  1< m_0r \ll \exp(\varpi \sqrt{\log X}),
\end{align*}
where $\varpi>0$ is a sufficiently small constant.
Regime \uppercase\expandafter{\romannumeral2} corresponds to
\begin{align*}
  \exp(\varpi \sqrt{\log X}) \ll  m_0r \ll X^{\frac{2}{15}},
\end{align*}
and Regime \uppercase\expandafter{\romannumeral3} corresponds to
\begin{align*}
  X^{\frac{2}{15}} \ll  m_0r \ll MX^{1/2+\varepsilon}.
\end{align*}
We then write
\begin{equation}\label{Regime}
  \mathcal{R}_2=\mathcal{E}_1 + \mathcal{E}_2 + \mathcal{E}_3 + O\bigg(X^{1-\varepsilon }\prod_{\mathcal{L}^2\leq p \leq \exp((\log \mathcal{L})^2)}\bigg(1+2\frac{b(p)}
  {\sqrt{p}}+b(p)^2\bigg)\bigg).
\end{equation}

\subsection{Evaluation of $\mathcal{E}_1$}

We first bound $\mathcal{E}_1$, which is precisely the contribution of Regime I. We have
\begin{equation}\label{eq: Reg1 defn of E1}
  \begin{split}
  \mathcal{E}_1=&\frac{2}{\left( 1-\frac{1}{\sqrt{2}}\right)^2} \sum_{q\leq M} b(q)^2 \sum_{\substack{g\leq \frac{M^2}{q^2}\\ (g,q)=1}} \frac{{\mu (g)}^2b(g)}{ \sqrt{g}} \sum_{\substack{m_0 \leq \frac{M^2}{gq^2} \\ (m_0,qg)=1}} a_q(m_0g)b(m_0) \sum_{\substack{r \leq X^{\frac{1}{2}+\varepsilon } \\ (r,2gm_0)=1 \\ 1<m_0r\ll \exp(\varpi \sqrt{\log x}) \\ }}
   \frac{{\mu (r)}^2}{ \sqrt{r}}   \\
  &\times \frac{1}{4} \sum_{\gamma \in \{\pm 1, \pm 2\}} \frac{1}{2\pi i} \int_{\frac{1}{\log X}-i(\log X)^2}^{\frac{1}{\log X}+i(\log X)^2} \frac{\Gamma(\frac{s}{2}+\frac{1}{4})}{\Gamma(\frac{1}{4})} \left( 1-\frac{1}{2^{\frac{1}{2}-s}}\right)\left(1 - \frac{1}{2^{1+2s}}\right)\zeta(1+2s)\left(\frac{X}{\pi}\right)^{\frac{s}{2}}e^{s^2} \left( gr\right)^{-s} \\
  &\times \sum_{n} \Lambda (n) \Phi_s\left(\frac{n}{X}\right)  \chi_{\gamma m_0 r}(n) \frac{\dd s}{s}.
  \end{split}
\end{equation}
By partial summation formula and the definition of $\Phi_s \left( X\right)$,
\begin{equation}\label{eq:  partial summation before pnt}
  \begin{split}
  \sum_n \Lambda(n) \Phi_s \left( \frac{n}{X}\right) \chi_{\gamma m_0r}(n)
  = -\int_0^{\infty} \frac{1}{X}\Phi_s'\left(\frac{w}{X}\right) \Bigg(\sum_{n\leq w}\Lambda(n) \chi_{\gamma m_0r}(n)\Bigg) \,\dd w.
  \end{split}
\end{equation}
As a well-known result, see \cite[equation (8) of Chapter 20]{Da00}, we have
\begin{align}\label{eq:  davenport pnt}
  \sum_{n\leq w}\Lambda(n) \chi_{\gamma m_0r}(n) = -\frac{w^{\beta_1}}{\beta_1} +O\left(w \exp(-c_2 \sqrt{\log w}) \right),
  \end{align}
  where the term $-w^{\beta_1}/\beta_1$ only appears if $L(s,\chi_{\gamma m_0r})$ has a real zero $\beta_1$ which satisfies $\beta_1 > 1- \frac{c_1}{\log |\gamma m_0r|}$ for some sufficiently small constant $c_1>0$, and $c_2>0$ is a absolute constant depends on $c_1$.
  All the constants in \eqref{eq:  davenport pnt}, implied or otherwise, are effectively.
  Finally we can control the contribution of this possible zero point.

The contribution from the error term in \eqref{eq:  davenport pnt} is easy to control, since
  \begin{align}\label{eq:  bound on phi ' s}
    \int_0^{\infty } \frac{1}{X}\left| \Phi_s'\left( \frac{w}{X}\right)\right|\,\dd w \ = \int_0^{\infty} |\Phi_s'(u)|\,\dd u \ \ll \ |s|+\log X
    \end{align}
  uniformly in s with Re($s$)$=\frac{1}{\log X}$ is bounded.
From Titchmarsh's book \cite[Theorem 3.5 and (3.11.8)]{Ti86}, we have $\zeta (1+z) \ll \log \vert $Im($z$)$\vert $ for Re($z$)$\geqslant -c/\log \vert$Im($z$)$\vert $ and $\vert$Im($z$)$\vert $$\geqslant 1$. This yields
\begin{equation}\label{eq: bound on other integral terems}
  \begin{split}
  \frac{\Gamma(\frac{s}{2}+\frac{1}{4})}
  {\Gamma(\frac{1}{4})}\zeta(1+2s) \ll \log \log X,
       \\
\left( 1-\frac{1}{2^{\frac{1}{2}-s}}\right)\left(1 - \frac{1}{2^{1+2s}}\right)
\left(\frac{X}{\pi}\right)
^{\frac{s}{2}}e^{s^2} \left( gr\right)^{-s} \ll 1 .
  \end{split}
\end{equation}
To bound the contribution of error term of  \eqref{eq:  davenport pnt}, we essentially need to show the following
\begin{equation*}\label{eq: sum on q,g,m0 ordinary zeros}
  \begin{split}
    \sum_{q}b(q)^2\sum_{(g, q)=1}\frac{b(g)d(g)}{\sqrt{g}}\sum_{m_0 \ll exp(\varpi \sqrt{\log X})}b(m_0)d(m_0)
  \end{split}
\end{equation*}
is small enough.
By elementary estimate, inserting the definition of $b(\cdot)$, we have
\begin{equation}\label{eq:sum on m0 ordinary zeros}
  \begin{split}
    \sum_{m_0 \ll exp(\varpi \sqrt{\log X})}b(m_0)d(m_0) \ll \sqrt{\log X}\exp(\varpi \sqrt{\log X}).
  \end{split}
\end{equation}
Also, we have
\begin{align}\label{eq:sum on q,g ordinary zeros}
  \sum_{q}b(q)^2\sum_{(g, q)=1}\frac{b(g)d(g)}{\sqrt{g}} \leq \prod_{\mathcal{L}^2\leq p \leq \exp((\log \mathcal{L})^2)}\bigg(1+2\frac{b(p)}{\sqrt{p}}+b(p)^2\bigg).
\end{align}

Taking \eqref{eq: Reg1 defn of E1}, \eqref{eq:  partial summation before pnt}, \eqref{eq: bound on phi ' s}, \eqref{eq: bound on other integral terems}, \eqref{eq:sum on m0 ordinary zeros} and \eqref{eq:sum on q,g ordinary zeros} together,
we see the error term of \eqref{eq:  davenport pnt} contributes
\begin{equation*}\label{eq: regime I ordinary zeros}
  \ll X \exp((c_3\varpi-c_2)\sqrt{\log X})\prod_{\mathcal{L}^2\leq p \leq \exp((\log \mathcal{L})^2)}\bigg(1+2\frac{b(p)}{\sqrt{p}}+b(p)^2\bigg)
\end{equation*}
to $\mathcal{E}_1$, where $c_3>0$ is a absolute constant; we have used the supposition on Regime I that $m_0r \ll $exp($\varpi \sqrt{\log X}$). Choose $\varpi$ small enough and fix it, we can write $c_4=c_2-c_3\varpi >0$ is a constant.

Recall that the conductor of the real primitive character $\chi_{\gamma m_0r}$ is $\ll \exp(\varpi \sqrt{\log X}) \leq \exp(2\varpi \sqrt{\log X})$. We apply Page's theorem \cite[equation (9) of Chapter 14]{Da00}, which implies that, for some fixed absolute constant $c_5>0$, there is at most one real primitive character $\chi_{\gamma m_0r}$ with modulus $\leq \exp(2\varpi \sqrt{\log X})$ for which $L(s,\chi_{\gamma m_0r})$ has a real zero satisfying
\begin{align}\label{eq:  lower bound for exceptional beta 1}
\beta_1 > 1 - \frac{c_5}{2\varpi \sqrt{\log X}}.
\end{align}

To estimate the contribution of the possible term $ -\frac{w^{\beta_1}}{\beta_1}$, we evaluate the integral
\begin{align*}
\int_0^{\infty} \frac{w^{\beta_1}}{\beta_1}
\frac{1}{X}\Phi_s'\left(\frac{w}{X}\right) \,\dd w
\end{align*}
arising from \eqref{eq: partial summation before pnt} and \eqref{eq:  davenport pnt}. By changing variable $\frac{w}{X}\mapsto u$ and integrating by parts, we can see that this integral equals
\begin{align*}
 X^{\beta_1}\int_0^{\infty} \frac{u^{\beta_1}}{\beta_1}\Phi_s'(u) \,\dd u \ = \ -X^{\beta_1}\int_0^{\infty}\Phi_s(u)
 u^{\beta_1-1}\,\dd u \ = \ -X^{\beta_1} \check\Phi\left( \frac{s}{2}+\beta_1-1\right),
\end{align*}
where $\check\Phi(\cdot)$ is defined in \eqref{def:checkPhi} and $\check\Phi\left( \frac{s}{2}+\beta_1-1\right)=O(1)$
for bounded $\Re s$.

We assume that a real zero satisfying \eqref{eq:  lower bound for exceptional beta 1} does exist, for otherwise we already have an acceptable bound for $\mathcal{E}_1$. Let $k^*$ be such that $\chi_{k^*}$ is the exceptional character with a real zero $\beta_1$ satisfying \eqref{eq:  lower bound for exceptional beta 1}. Then we have
\begin{equation}\label{eq: handle siegel zeros 1}
  \begin{split}
  \mathcal{E}_1=&-\frac{2}{\left( 1-\frac{1}{\sqrt{2}}\right)^2}\frac{1}{2\pi i}\frac{1}{4} X^{\beta_1} \sum_{q\leq M} b(q)^2 \sum_{\substack{ g\leq \frac{M^2}{q^2}\\ (g,q)=1}} \frac{b(g) }{ \sqrt{g}} \sum_{\substack{ m_0 \leq \frac{M^2}{gq^2} \\ (m_0,qg)=1}} a_q(m_0g)b(m_0)  \\
  &\times \sum_{\substack{r \leq X^{\frac{1}{2}+\varepsilon } \\ (r,2gm_0)=1 \\ m_0r\ll \exp(\varpi \sqrt{\log x}) \\ \gamma^*m_0r=k^*}} \frac{\mu (r)^2}{\sqrt{r}} \int_{\frac{1}{\log X}-i(\log X)^2}^{\frac{1}{\log X}+i(\log X)^2} \frac{\Gamma(\frac{s}{2}+\frac{1}{4})}{\Gamma(\frac{1}{4})} \left( 1-\frac{1}{2^{\frac{1}{2}-s}}\right)\left(1 - \frac{1}{2^{1+2s}}\right) \zeta(1+2s)\left(\frac{X}{\pi}\right)^{\frac{s}{2}}   \\
  &\times e^{s^2}\left( gr\right)^{-s} \check\Phi\left( \frac{s}{2}+\beta_1-1\right) \frac{\dd s}{s}+ O\left(X \exp(-c_4\sqrt{\log X})\prod_{\mathcal{L}^2\leq p \leq \exp((\log \mathcal{L})^2)}\bigg(1+2\frac{b(p)}
  {\sqrt{p}}+b(p)^2\bigg)\right),
  \end{split}
\end{equation}
where $\gamma^* \in \{ \pm 1, \pm 2\} $ is fixed. In fact, there is only one choice of $\gamma$ that can give rise to the exceptional character since $m_0r$ is odd and positive. So the notation $\gamma^*$ makes sense.

We handle the $s$-integral in \eqref{eq: handle siegel zeros 1} by moving the line of integration to $\text{Re}(s) = -\frac{c_5}{\log \log X}$, where $c_5>0$ is small enough that  $\zeta (1+z) \ll \log \vert $Im($z$)$\vert $ for Re($z$)$\geqslant -c_5/\log \vert$Im($z$)$\vert $ and $\vert$Im($z$)$\vert $$\geqslant 1$ (see Titchmarsh \cite[Theorem 3.5 and Eq. (3.11.8)]{Ti86}).
We estimate the integral on the line $\text{Im}(s)=(\log X)^2$ with trivial estimates as follow
\begin{equation*}
  \begin{split}
  &\frac{\Gamma(\frac{s}{2}+\frac{1}{4})}{\Gamma(\frac{1}{4})}\zeta(1+2s) \ll \log \log X,\\
  &\left( 1-\frac{1}{2^{\frac{1}{2}-s}}\right)\left(1 - \frac{1}{2^{1+2s}}\right)\left(\frac{X}{\pi}\right)^{\frac{s}{2}}e^{s^2} \left( gr\right)^{-s} \ll 1 ,
  \end{split}
\end{equation*}
and so does $\text{Im}(s)=-(\log X)^2$.
We also estimate the integral on the line $\Re(s)=-\frac{c_5}{\log\log X}$ trivially, yields
\begin{equation*}
  \begin{split}
  &\frac{\Gamma(\frac{s}{2}+\frac{1}{4})}{\Gamma(\frac{1}{4})}\zeta(1+2s)X^{s/2} \ll \log \log X\exp\bigg(-c_5/2 \frac{\log X}{\log\log X}\bigg)  \\
  &\left( 1-\frac{1}{2^{\frac{1}{2}-s}}\right)\left(1 - \frac{1}{2^{1+2s}}\right)\left(\frac{1}{\pi}\right)^{\frac{s}{2}}e^{s^2} \left( gr\right)^{-s} \ll 1 .
  \end{split}
\end{equation*}
After moving the line of integration, a residue appears which comes from the pole at $s=0$.
Therefore,
\begin{equation}\label{eq: move the line of integration}
  \begin{split}
    \int_{\frac{1}{\log X}-i(\log X)^2}^{\frac{1}{\log X}+i(\log X)^2} =2\pi i \text{Res}_{s=1}(\cdot) +O\bigg(\exp(-c_6 \frac{\log X}{\log\log X})\bigg) ,
  \end{split}
\end{equation}
where $0<c_6<c_5/2$ is a constant and the residue is bounded by $O(\log X)$ as in equation \eqref{eq:residue}.

Next, let's treat the essential part of this multiple sum by dividing the range of $\frac{k^*}{\gamma^* }$ into two segments. If $\frac{k*}{\gamma* }\geq \mathcal{L}^2 $, then
\begin{equation}\label{eq: handle parts except for the integral 1}
  \begin{split}
  &\ \sum_{q}b(q)^2\sum_{(g, q)=1}\frac{b(g)d(g)}
  {\sqrt{g}}\sum_{\substack{m_0, r\\ m_0r \ll exp(\varpi \sqrt{\log X})\\ \gamma^*m_0r=k^* }} \frac{b(m_0)d(m_0)\mu(r)^2}{\sqrt{r}} \\
  &\ll  \prod_{\substack{p'<\mathcal{L}^2 \  or \ p'>\exp((\log \mathcal{L} )^2) \\ p' | \frac{k^*}{\gamma^* } }}\frac{1}{\sqrt{p'}} \prod_{\substack{\mathcal{L}^2 \leq p'' \leq \exp((\log \mathcal{L} )^2)\\ p'' | \frac{k^*}{\gamma^* }}} \bigg(\frac{2}{\log p''}+\frac{1}{\sqrt{p''}}\bigg) \prod_{\mathcal{L}^2\leq p \leq \exp((\log \mathcal{L})^2)}\bigg(1+2\frac{b(p)}{\sqrt{p}}+b(p)^2\bigg) \\
  &\ll (\log \mathcal{L} )^{-1} \prod_{\mathcal{L}^2\leq p \leq \exp((\log \mathcal{L})^2)}\bigg(1+2\frac{b(p)}{\sqrt{p}}+b(p)^2\bigg).
  \end{split}
\end{equation}
If $\frac{k^*}{\gamma^* } < \mathcal{L}^2 $,
\begin{equation}\label{eq: handle parts except for the integral 2}
  \begin{split}
  &\ \sum_{q}b(q)^2\sum_{(g, q)=1}\frac{b(g)d(g)}{\sqrt{g}}
  \sum_{\substack{m_0, r\\ m_0r \ll exp(\varpi \sqrt{\log X})\\ \gamma^*m_0r=k^* }} \frac{b(m_0)d(m_0)\mu(r)^2}{\sqrt{r}} \\
  &\ll  \prod_{\substack{p'<\mathcal{L}^2  \\ p' | \frac{k^*}{\gamma^* } }}\frac{1}{\sqrt{p'}} \prod_{\mathcal{L}^2\leq p \leq \exp((\log \mathcal{L})^2)}\bigg(1+2\frac{b(p)}{\sqrt{p}}+b(p)^2\bigg)\\
  &\ll \prod_{\mathcal{L}^2\leq p \leq \exp((\log \mathcal{L})^2)}\bigg(1+2\frac{b(p)}{\sqrt{p}}+b(p)^2\bigg).
  \end{split}
\end{equation}
Up to now, inserting \eqref{eq: move the line of integration}, \eqref{eq: handle parts except for the integral 1} and \eqref{eq: handle parts except for the integral 2} into \eqref{eq: handle siegel zeros 1}, we have
\begin{equation*}
  \begin{split}
\mathcal{E}_1 \ll &\left(X^{\beta_1}\log X \left(1+(\log \mathcal{L})^{-1} \right)  + X \exp(-c_4\sqrt{\log X})\right) \\
&\times \prod_{\mathcal{L}^2\leq p \leq \exp((\log \mathcal{L})^2)}\bigg(1+2\frac{b(p)}{\sqrt{p}}+b(p)^2\bigg).
  \end{split}
\end{equation*}
Now recall that we have the bound
\begin{align}\label{eq: upper bound for beta 1}
  \beta_1 < 1 - \frac{c_7}{\sqrt{|k^*|} (\log |k^*|)^2},
\end{align}
where $c_7 > 0$ is a absolute constant (see \cite[equation (12) of Chapter 14]{Da00}).
If $\frac{k^*}{\gamma^*}$ satisfies $|\frac{k^*}{\gamma^*}| <  \mathcal{L}^2  $ then by \eqref{eq:  upper bound for beta 1} we derive
\begin{align}\label{eq: when modulus small}
  X^{\beta_1} \ll X \exp(-c_8 (\log X)^{\frac{1}{2}-\varepsilon}),
  \end{align}
where $c_8>0$ is a absolute constant.
And, inserting \eqref{eq: move the line of integration}, \eqref{eq: handle parts except for the integral 2} and \eqref{eq: when modulus small} into \eqref{eq: handle siegel zeros 1}, by trivial estimates we obtain
\begin{align*}
  \mathcal{E}_1 \ll  X \exp(-c_8 (\log X)^{1/4})\prod_{\mathcal{L}^2\leq p \leq \exp((\log \mathcal{L})^2)}\bigg(1+2\frac{b(p)}{\sqrt{p}}+b(p)^2\bigg).
\end{align*}
If $\frac{k^*}{\gamma^*} \geq \mathcal{L}^2 $, inserting \eqref{eq:  lower bound for exceptional beta 1}, \eqref{eq: move the line of integration} and \eqref{eq: handle parts except for the integral 1} into \eqref{eq: handle siegel zeros 1}, we deduce by trivial estimation
\begin{align*}
  \mathcal{E}_1 \ll X(\log X)^{-2023}\prod_{\mathcal{L}^2\leq p \leq \exp((\log \mathcal{L})^2)}\bigg(1+2\frac{b(p)}{\sqrt{p}}+b(p)^2\bigg).
\end{align*}
This completes the proof of the bound for $\mathcal{E}_1$.

\subsection{Evaluation of $\mathcal{E}_2$}
Now we are going to bound $\mathcal{E}_2$ in \eqref{Regime}, which is the contribution of Regime \uppercase\expandafter{\romannumeral2}. Comparing with \eqref{eq: pretreatment of R2 2}, we see that we may write $\mathcal{E}_2$ as
\begin{equation*}
\begin{split}
&\frac{2}{\left( 1-\frac{1}{\sqrt{2}}\right)^2}
\sum_{q\leq M}{b(q)}^2\frac{1}{2 \pi i} \int_{\frac{1}{\log X}-i(\log X)^2}^{\frac{1}{\log X}+i(\log X)^2} K(s) e^{s^2}\sum_{\substack{g \leq \frac{M^2}{q^2} \\
(g, 2q)=1}} \frac{b(g)}{g^{1/2+s}} \sum_{\substack{m_0 \leq \frac{M^2}{q^2g}\\
(m_0, q g)=1}} a_q(m_0g)b(m_0) \\
&\times \sum_{\substack{r \leq X ^{1 / 2+\varepsilon} \\
(r, 2 gm_0)=1\\
  \exp(\varpi \sqrt{\log X}) \ll  m_0r \ll X^{\frac{2}{15}}}} \frac{\mu(r)^2}{r^{1 / 2+s}} \sum_{\gamma \in\{ \pm 1, \pm 2\}} \sum_n \Lambda(n) \Phi_s(n / X) \chi_{\gamma m_0 r}(n) \dd  s,
\end{split}
\end{equation*}
where
\begin{equation*}
K(s) = \frac{\Gamma(\frac{s}{2}+\frac{1}{4})}{\Gamma(\frac{1}{4})} \left( 1-\frac{1}{2^{\frac{1}{2}-s}}\right)\left(1 - \frac{1}{2^{1+2s}}\right)\zeta(1+2s)\left(\frac{X}{\pi}\right)^{\frac{s}{2}}.
\end{equation*}
Similar to \eqref{eq: bound on other integral terems},we have $K(s)\ll \log \log X \ll \log X $. We apply the triangle inequality and take a supremun in $s$ to see that, for some complex number $s_0$ satisfying $\Re(s_0)=1/\log X$, $|\Im(s_0)|\leq (\log X)^2$ we have
\begin{align}\label{Reg2ineq1}
\mathcal{E}_2 &\ll
(\log X)^3\sum_{q\leq M}{b(q)}^2
\sum_{\substack{g \leq \frac{M^2}{q^2} \\
(g, 2q)=1}} \frac{b(g)d(g)}{g^{1/2}}
\sum_{\substack{m_0 \leq \frac{M^2}{q^2g}\\
(m_0, q g)=1}} d(m_0)b(m_0) \nonumber\\
&\times \sum_{\substack{r \leq X ^{1 / 2+\varepsilon} \\
(r, 2 gm_0)=1 \\
\exp(\varpi \sqrt{\log X}) \ll  m_0r \ll X^{\frac{2}{15}}}} \frac{\mu(r)^2}{r^{1 / 2}} \sum_{\gamma \in\{ \pm 1, \pm 2\}}\Big |\sum_n \Lambda(n) \Phi_{s_0}(n / X) \chi_{\gamma m_0 r}(n)\Big|\nonumber\\
&\ll (\log X)^3 \prod_{p} \Big(1+2\frac{b(p)}{\sqrt{p}}+b(p)^2\Big)
\sum_{m_0 \leq M^2} d(m_0)b(m_0)\nonumber\\
&\times \sum_{\substack{r \leq X ^{1 / 2+\varepsilon} \\
(r,2m_0)=1\\
\exp(\varpi \sqrt{\log X}) \ll  m_0r \ll X^{\frac{2}{15}}}} \frac{\mu(r)^2}{r^{1 / 2}} \sum_{\gamma \in\{ \pm 1, \pm 2\}}\Big |\sum_n \Lambda(n) \Phi_{s_0}(n / X) \chi_{\gamma m_0 r}(n)\Big|.
\end{align}
For the first inequality we use $a_q(m_0g)\leq d(m_0g)\leq d(m_0)d(g)$. Now we write $a=m_0r$ and extend sums in (\ref{Reg2ineq1}) properly, since each of terms is non-nagative, we have
\begin{align}\label{Reg2ineq2}
\mathcal{E}_2&\ll
(\log X)^3 \prod_{p} \Big(1+2\frac{b(p)}{\sqrt{p}}+b(p)^2\Big)\nonumber \\
&\times \sum_{\exp(\varpi \sqrt{\log X}) \ll  a \ll X^{\frac{2}{15}}}\frac{\sum_{\substack{m_0\mid a}}d(m_0)b(m_0)\sqrt{m_0}}{a^{1/2}}\sum_{\gamma \in\{ \pm 1, \pm 2\}}\Big |\sum_n \Lambda(n) \Phi_{s_0}(n / X) \chi_{\gamma a}(n)\Big|.
\end{align}
Obviusly, we have
\begin{align}\label{Reg2ele}
\sum_{\substack{m_0\mid a}}d(m_0)b(m_0)\sqrt{m_0}&=\prod_{\substack{p\mid a}}(1+2b(p)\sqrt{p}) \nonumber\\
&\leq\prod_{\substack{p\mid a\\{\mathcal{L}}^2\leq p\leq\exp((\log \mathcal{L})^2)}}\bigg(1+\frac{\mathcal{L}}{\log \mathcal{L}}\bigg) \nonumber\\
&\leq \bigg(1+\frac{\mathcal{L}}{\log \mathcal{L}}\bigg)^{\omega_\mathcal{L}(a)},
\end{align}
where
\begin{equation*}
 \omega_\mathcal{L}(a):=\sum_{\substack{p\mid a\\{\mathcal{L}}^2\leq p\leq\exp((\log \mathcal{L})^2)}} 1.
\end{equation*}
Applying the `record-breaking' argument  (\cite[Theorem 2.9]{MV07}) to $\omega_\mathcal{L}(a)$, we find a trivial bound $\omega_\mathcal{L}(a)\leq \frac{\log a}{2\log \mathcal{L}}$.
Inserting this bound to (\ref{Reg2ele}), by a little caculation, it's easy to show, for any $A>0$,
\begin{equation}\label{Reg2ele2}
\bigg(1+\frac{\mathcal{L}}{\log \mathcal{L}}\bigg)^{\omega_\mathcal{L}(a)}\leq \frac{a^{1/2}}{(\log a)^A},
\end{equation}
which is valid for $a \gg \exp (\varpi \sqrt{\log X})$ and $X$ sufficient large. We apply (\ref{Reg2ele})(\ref{Reg2ele2}) to (\ref{Reg2ineq2}), and by a little abuse of letters, we write $q=\gamma a$. After breaking the range of $q$ into dyadic segments, we find
\begin{equation*}\label{EpsilonQ}
 \mathcal{E}_2\ll
 (\log X)^3 \prod_{p} \Big(1+2\frac{b(p)}{\sqrt{p}}+b(p)^2\Big)
 \sum_{\substack{Q=2^j \\ \exp(\varpi \sqrt{\log X}) \ll Q \ll X^{\frac{2}{15}}}} \mathcal{E}(Q),
\end{equation*}
where
\begin{equation*}
\mathcal{E}(Q):={(\log Q)}^{-A} \sum_{\chi \in S(Q)}\Big|\sum_n \Lambda(n) \Phi_{s_0}(n / X) \chi(n)\Big|.
\end{equation*}
To bound $\mathcal{E}_2$ in Regime \uppercase\expandafter{\romannumeral2}  it may suffices to show, for some absolute constant $B>0$ large enough, that
\begin{equation}\label{Reg2resu}
\mathcal{E}(Q)\ll \frac{X}{(\log X)^B}.
\end{equation}
In Regime \uppercase\expandafter{\romannumeral2} we employ zero-density estimates, the proof is similar to \cite[Section 6.4]{BP22}. We begin by writing $\Phi_{s_0}$ as the integral of its Mellin transform, yielding
\begin{equation*}
\sum_n \Lambda(n) \Phi_{s_0}(n / X) \chi(n)=\frac{1}{2 \pi i} \int_{(2)} X^w \check{\Phi}\left(w+s_0 / 2\right)\left(-\frac{L^{\prime}}{L}(w, \chi)\right) \dd w.
\end{equation*}
Observe that from repeated integration by parts we have
\begin{equation}\label{rapdeca}
\left|\check{\Phi}\left(\sigma+i t+s_0 / 2\right)\right| \ll_{\sigma, j}(\log X)^j\left(1+\left|t-\frac{\operatorname{Im}\left(s_0\right)}{2}\right|\right)^{-j}
\end{equation}
for every non-negative integer $j$.
We shift the line of integration to $\operatorname{Re}(w)=-1 / 2$, leaving residues from all of the zeros of $L(w, \chi)$ in the critical strip. We bound the new integral by applying the estimate
\begin{equation*}
\left|\frac{L^{\prime}}{L}(w, \chi)\right| \ll \log (q|w|)
\end{equation*}
valid for $\operatorname{Re}(w)=-1 / 2$, and deduce that
\begin{equation*}
\sum_n \Lambda(n) \Phi_{s_0}(n / X) \chi(n)=-\sum_{\substack{L(\rho, \chi)=0 \\ 0 \leq \beta \leq 1}} X^\rho \check{\Phi}\left(\rho+s_0 / 2\right)+O\left(\frac{(\log X)^{O(1)}}{X^{1 / 2}}\right) .
\end{equation*}
We have written here $\rho=\beta+i \gamma$. The error term is, of course, completely acceptable for (\ref{Reg2resu})when summed over $\chi \in S(Q)$.

By (\ref{rapdeca}), the contribution to $\mathcal{E}(Q)$ from those $\rho$ with $|\gamma|>Q^{1 / 2}$ is $\ll X Q^{-100}$, say, and this gives an acceptable bound. We have therefore obtained
\begin{equation}\label{zero des}
\mathcal{E}(Q)\ll X\exp (-\varpi \sqrt{\log X})+{(\log Q)}^{-A}\sum_{\chi \in S(Q)}\sum_{\substack{L(\rho, \chi)=0 \\ 0 \leq \beta \leq 1\\ |\gamma|\leq Q^{1/2}}} X^{\beta}.
\end{equation}
In order to bound the right side of (\ref{zero des}), we first need to introduce some notation. For a primitive Dirichlet character $\chi$ modulo $q$, let $N(T, \chi)$ denote the number of zeros of $L(s, \chi)$ in the rectangle
\begin{equation*}
0 \leq \beta \leq 1, \quad|\gamma| \leq T.
\end{equation*}
For $T \geq 2$, say, we have \cite[chapter 12]{Da00}
\begin{equation}\label{zero est}
N(T, \chi) \ll T \log (q T).
\end{equation}
For $1 / 2 \leq \alpha \leq 1$, define $N(\alpha, T, \chi)$ to be the number of zeros $\rho=\beta+i \gamma$ of $L(s, \chi)$ in the rectangle
\begin{equation*}
\alpha \leq \beta \leq 1, \quad|\gamma| \leq T,
\end{equation*}
and define
\begin{equation*}
N(\alpha, Q, T)=\sum_{q \leq Q} {\sum_{\chi(\bmod q)}}^{*} N(\alpha, T, \chi) .
\end{equation*}
The summation over $\chi$ is over primitive characters. We shall employ Montgomery's zero-density estimate \cite[Theroem 1]{Mo69}
\begin{equation}\label{Mon's est}
N(\alpha, Q, T) \ll(Q^2 T)^{\frac{5}{2}(1-\alpha)}{(\log QT)}^{13},
\end{equation}
which holds for $\alpha \geq 4 / 5$.\\
In (\ref{zero des}), we separate the zeros $\rho$ according to whether $\beta<4 / 5$ or $\beta \geq 4 / 5$. Using (\ref{zero est}) we deduce
\begin{equation}\label{beta<}
{(\log Q)}^{-A} \sum_{\chi \in S(Q)} \sum_{\substack{L(\rho, \chi)=0 \\ 0 \leq \beta<4 / 5 \\|\gamma| \leq Q^{1 / 2}}} X^\beta \ll {(\log Q)}^{-A+1}X^{4 / 5} Q^{3/2}.
\end{equation}
For those zeros with $\beta \geq 4 / 5$ we write
\begin{equation*}
X^\beta=X^{4 / 5}+(\log X) \int_{4 / 5}^\beta X^\alpha \dd  \alpha.
\end{equation*}
We then embed $S(Q)$ into the set of all primitive characters with conductors $\leq Q$. Applying (\ref{zero est}) and (\ref{Mon's est}), we obtain
\begin{align*}
\sum_{\chi \in S(Q)} \sum_{\substack{L\left(\rho, \chi_q\right)=0 \\ 4 / 5 \leq \beta \leq 1 \\|\gamma| \leq Q^{1 / 2}}} X^\beta & \ll X^{4 / 5} Q^{3 / 2}\log Q+(\log X) \int_{4 / 5}^1 X^\alpha N\left(\alpha, Q, Q^{1 / 2}\right) \dd \alpha \\
& \ll X^{4 / 5} Q^{3 / 2}\log Q+ {(\log Q)}^{15}\int_{4 / 5}^1 X^\alpha Q^{\frac{25}{4}(1-\alpha)} \dd \alpha .
\end{align*}
Since $Q \ll X^{2 / 15}$, the integrand of this latter integral is maximized when $\alpha=1$. It follows that
\begin{equation}\label{qfinal}
{(\log Q)}^{-A}\sum_{\chi \in S(Q)} \sum_{\substack{L\left(\rho, \chi_q\right)=0 \\ 4 / 5 \leq \beta \leq 1 \\|\gamma| \leq Q^{1 / 2}}} X^\beta \ll X^{4 / 5} Q^{3 / 2}{(\log Q)}^{-A+1}+ X{(\log Q)}^{-A+15}\ll X{(\log Q)}^{-A+15}.
\end{equation}
Combining (\ref{qfinal}), (\ref{beta<}), and (\ref{zero des}) yields
\begin{equation*}
\mathcal{E}(Q) \ll X{(\log Q)}^{-A+15}+ X \exp (-\varpi \sqrt{\log X}),
\end{equation*}
and this suffices for (\ref{Reg2resu}) when $A$ large enough.

\subsection{Evaluation of $\mathcal{E}_3$}
In Regime \uppercase\expandafter{\romannumeral3}
we have $X^{2/15}\ll Q\ll M^2X^{\frac{1}{2}+\varepsilon}$.
Here we use Vaughan's identity
\begin{equation*}\label{eqn:vaughan'sidentity}
  \Lambda(n)=\Lambda_{\leq V}(n)+
  (\mu_{\leq V}\star \log)(n)-
  (\mu_{\leq V}\star \Lambda_{\leq V}\star 1)(n)+
  (\mu_{> V}\star \Lambda_{> V}\star 1)(n),
\end{equation*}
where $V^2\leq X$, $\star$ is the Dirichlet convolution,
for an arithmetic $f$ we denote $f_{\leq V}(n)=f(n)$ for $n\leq V$ and zero for otherwise, and $f_{>V}(n)=f(n)-f_{\leq V}(n)$.
For $n\asymp X$, $\Lambda_{\leq V}(n)=0$, this reduces the estimation of three different sums, and this have done by Baluyot and Pratt\cite{BP22}.

\begin{proposition}\label{prop:regime3}
With the same notions as before, we have
\begin{equation*}\label{eqn:Baluyot_Pratt1}
    \sum_{\chi \in S(Q)}
    \Big|\sum_n (\mu_{\leq V}\star \log)(n) \Phi_{s_0}(n / X) \chi(n)\Big| \ll
    Q^{\frac{3}{2}}VX^{\varepsilon},
\end{equation*}
\begin{equation*}\label{eqn:Baluyot_Pratt2}
    \sum_{\chi \in S(Q)}
    \Big|\sum_n (\mu_{\leq V}\star \Lambda_{\leq V}\star 1)(n)
    \Phi_{s_0}(n / X) \chi(n)\Big| \ll
    Q^{\frac{3}{2}+\varepsilon}V^2,
\end{equation*}
and
\begin{equation*}\label{eqn:Baluyot_Pratt3}
    \sum_{\chi \in S(Q)}
    \Big|\sum_n (\mu_{> V}\star \Lambda_{> V}\star 1)(n)
    \Phi_{s_0}(n / X) \chi(n)\Big| \ll
    Q^{\frac{1}{2}}X^{\varepsilon}
    (MX^{\frac{3}{4}}+XV^{-\frac{1}{2}}
    +XQ^{-\frac{1}{2}}).
\end{equation*}
\end{proposition}
\begin{proof}
  See \cite{BP22} (6.5.2)--(6.5.5).
\end{proof}

Similar to \eqref{Reg2ineq2}, we have
\begin{multline*}
 \mathcal{E}_3\ll
 \prod_{p} \Big(1+2\frac{b(p)}{\sqrt{p}}+b(p)^2\Big)
 \sum_{\substack{Q=2^j \\ X^{2/15}\ll Q \ll M^2 X^{1 / 2+\varepsilon}}} M^{\frac{3}{2}+\varepsilon}Q^{-\frac{1}{2}}
 \sum_{\chi \in S(Q)}\Big|\sum_n \Lambda(n) \Phi_{s_0}(n / X) \chi(n)\Big|
 \\ \ll
\left(M^{\frac{7}{2}}X^{\frac{1}{2}+3\varepsilon} V^2
+ M^{\frac{5}{2}} X^{\frac{3}{4}+2\varepsilon}
+M^{\frac{3}{2}}V^{-\frac{1}{2}}X^{1+2\varepsilon}
+ M^{\frac{3}{2}} X^{\frac{14}{15}+2\varepsilon}\right)
\prod_{p} \Big(1+2\frac{b(p)}{\sqrt{p}}+b(p)^2\Big).
\end{multline*}
First let $\theta\leq \frac{2}{45}-3\varepsilon$, then $M^{\frac{5}{2}} X^{\frac{3}{4}+2\varepsilon}+M^{\frac{3}{2}} X^{\frac{14}{15}+2\varepsilon}\ll X^{1-\varepsilon}$, then let $M^2V^{\frac{5}{2}}=X^{\frac{1}{2}}$,
i.e. $V=X^{\frac{1-4\theta}{5}}$,
thus
\begin{equation*}
  \mathcal{E}_3\ll \mathcal{M}_2X^{-\varepsilon}.
\end{equation*}
Then we have Proposition \ref{prop:R2}.

\end{proof}

\section{Concluding remarks}\label{remarks}

The study of the distribution of central values of automorphic $L$-functions in quadratic twists families is also active.
Radziwi{\l \l} and Soundararajan\cite{RS15} established an asymptotic formula for the first moment of central $L$-values
of quadratic twists of elliptic curves, and they used this result to derive a one-sided central limit theorem and unconditional upper bounds of small moments corresponding to Keating--Snaith conjecture\cite{KS00b}.
Shen\cite{Sh22} set up asymptotic formulas with small error term for the first moment of quadratic twists of modular $L$-functions and the first derivative of quadratic twists of modular $L$-functions.
Recently, Li\cite{Li22} got an asymptotic formula for the second moment of quadratic twists of modular $L$-functions, which was previously known conditionally on the Generalized Riemann Hypothesis by the work of Soundararajan and Young\cite{SY10}.
Combining Soundararajan's third moment method\cite{So00}, Young's recursion method\cite{Yo09} and Heath-Brown's large sieve for quadratic characters\cite{HB95}, the second named author and Huang\cite{HH22a} established asymptotic formulas for the twisted first moment of twisted $\GL(3)$ $L$-functions, and they employed these to determine the correspondence for the original Hecke--Maass cusp form or its dual form by central $L$-values.
Also recently, Gao and Zhao\cite{GZ23} built an asymptotic formula for the twisted first moment of central values of the product of a quadratic Dirichlet $L$-function and a quadratic twist of a modular $L$-function.

Hoffstein and Lockhart\cite{HL99} extended Heath-Brown's idea to prove an extreme central values result for quadratic twists of modular $L$-functions.
Recently, the second named author and Huang\cite{HH22c} got an extreme central $L$-values result for almost prime quadratic twists of an elliptic curve $E$, motivated by studying how large the extreme orders of the Tate--Shafarevich groups in the quadratic twist family of an elliptic curve are under the Birth--Swinnerton-Dyer conjecture. The restriction ``almost prime quadratic twists" is just used to control the size of the Tamagawa numbers.
Tamagawa number is bounded by the square of the divisor function trivially, which is much larger than both the conjectured order of magnititude of the extrema of the $L$-function and the order of magnititudethe of the extremum of the $L$-function that can be detected at present.

Selberg's central limit theorem says that for large $T$, when $t$ is chosen uniformly in $[T,2T]$, $\log\zeta(\frac{1}{2}+it)$ distributes like a complex Gaussian with mean 0 and variance $\log\log T$.
Keating-Snaith\cite{KS00a,KS00b}, based on Katz--Sarnak's philosophy\cite{KS99}, guessed similar distributions for families of $L$-functions and divided it into three classes of families: Unitary families, Symplectic families and Orthogonal families, which are corresponding to three classes of random matrices.
Keating--Snaith also conjectured the asymptotic formulas of moments of $L$-functions based on random matrix theory. Afterwards, Diaconu--Goldfeld--Hoffstein\cite{DGH03} and Conrey--Farmer--Keating--Rubinstein--Snaith\cite{CFKRS05} made it possible to speculate the complete main term for even moments.
By using these asymptotic formulas, it is also possible to do some researches on moments conjecture in quadratic twists families.
The techniques mainly come from studies of the Riemann $\zeta$-function in three types:
unconditional upper bounds for small moments\cite{BCR17,HRS19,HB81}, upper bounds assuming the Riemann Hypothesis\cite{Ha13,So09} and unconditional lower bounds for moments\cite{HS22,HB81,Mo71,RS13,Ra78,Ra80}.
Similar progress in quadratic twists families includes: unconditional upper bounds for small moments\cite{RS15}, upper bounds assuming Generalized Riemann Hypothesis\cite{GZ21a,MT14,Sono16,So09} and
unconditional lower bounds for moments\cite{GZ21b,HH22b,RS06}.
Recently, the second named author and Huang\cite{HH22b} gained unconditional lower bounds of moments of central values of quadratic twisted self-dual $\GL(3)$ $L$-functions based on the Heap--Soundararajan's lower bounds principle\cite{HS22}.
For progress on these topics in other families, see \cite{BFKMMS23,So21}.

\section*{Acknowledgements}
The authors would like to thank Professors Bingrong Huang, Jianya Liu, and Xianchang Meng for their help and encouragements.
The authors would like to thank the reviewers for their valuable comments.



\begin{thebibliography}{10}

\bibitem{BP22}
Baluyot, Siegfred; Pratt, Kyle.
Dirichlet $L$-functions of quadratic characters of prime conductor at the central point.
\emph{J. Eur. Math. Soc. (JEMS)} 24 (2022), no. 2, 369--460.

\bibitem{BCR17}
Bettin, Sandro; Chandee, Vorrapan; Radziwi{\l \l}, Maksym.
The mean square of the product of the Riemann zeta-function with Dirichlet polynomials.
\emph{J. Reine Angew. Math.} 729 (2017), 51--79.

\bibitem{BFKMMS23}
Blomer, Valentin; Fouvry, \'{E}tienne; Kowalski, Emmanuel; Michel, Philippe; Mili\'{c}evi\'{c}, Djordje; Sawin, Will.
The second moment theory of families of $L$-functions--the case of twisted Hecke $L$-functions.
\emph{Mem. Amer. Math. Soc.} 282(2023), no.1394.

\bibitem{BS17}
Bondarenko, Andriy; Seip, Kristian.
Large greatest common divisor sums and extreme values of the Riemann zeta function. \emph{Duke Math. J.} 166 (2017), no. 9, 1685--1701.

\bibitem{BS18}
Bondarenko, Andriy; Seip, Kristian.
Extreme values of the Riemann zeta function and its argument.
\emph{Math. Ann.} 372 (2018), no. 3-4, 999--1015.

\bibitem{Ch65}
Chowla, Sarvadaman.
The Riemann hypothesis and Hilbert's tenth problem.
\emph{Norske Vid. Selsk. Forh. (Trondheim) } 38 (1965), 62--64.

\bibitem{Co07}
Cohen, Henri.
Number theory. Vol. I. Tools and Diophantine equations.
Graduate Texts in Mathematics, 239. \emph{Springer, New York,} 2007.

\bibitem{CFKRS05}
Conrey, Brian; Farmer, David W.; Keating, Jonathan P.; Rubinstein, Michael O.; Snaith, Nina Claire.
Integral moments of $L$-functions.
\emph{Proc. London Math. Soc.} (3) 91 (2005), no. 1, 33--104.

\bibitem{DM24}
Darbar, Pranendu; Maiti, Gopal.
Large values of quadratic Dirichlet $L$-functions over monic irreducible polynomial in $\mathbb{F}_q[t]$.
\emph{Proc. Amer. Math. Soc.} 152 (2024), no. 8, 3243--3254.

\bibitem{DM25}
Darbar, Pranendu; Maiti, Gopal.
Large values of quadratic Dirichlet
$L$-functions.
\emph{Math. Ann.} (2025). https://doi.org/10.1007/s00208-025-03187-6

\bibitem{Da00}
Davenport, Harold.
Multiplicative number theory. Third edition. Revised and with a preface by Hugh L. Montgomery. Graduate Texts in Mathematics, 74.
\emph{Springer-Verlag, New York,} 2000.

\bibitem{dT19}
de la Bret\`eche, R\'egis; Tenenbaum, G\'erald Sommes de G\'al et applications.
\emph{Proc. Lond. Math. Soc.} (3) 119 (2019), no. 1, 104--134.

\bibitem{DGH03}
Diaconu, Adrian; Goldfeld, Dorian; Hoffstein, Jeffrey.
Multiple Dirichlet series and moments of zeta and $L$-functions.
\emph{Compositio Math.} 139 (2003), no. 3, 297--360.

\bibitem{DW21}
Diaconu, Adrian; Whitehead, Ian.
On the third moment of $L(\frac{1}{2},\chi_d)$ II: the number field case.
\emph{J. Eur. Math. Soc. (JEMS)} 23 (2021), no. 6, 2051--2070.

\bibitem{FGH07}
Farmer, David W.; Gonek, Steven Mark; Hughes, Christopher.
The maximum size of $L$-functions.
\emph{J. Reine Angew. Math.} 609 (2007), 215--236.

\bibitem{GZ21a}
Gao, Peng; Zhao, Liangyi.
Bounds for moments of quadratic Dirichlet $L$-functions of prime-related moduli.
\emph{Colloq. Math.} 171 (2023), no. 1, 61--77.

\bibitem{GZ21b}
Gao, Peng; Zhao, Liangyi.
Lower bounds for moments of quadratic twists of modular $L$-functions.
\emph{Funct. Approx. Comment. Math.} 71 (2024), no. 2, 183--200.

\bibitem{GZ23}
Gao, Peng; Zhao, Liangyi.
Twisted first moment of quadratic and quadratic twist $L$-functions.
\emph{Funct. Approx. Comment. Math.}
70 (2024), no. 1, 101--127.

\bibitem{Go79}
Goldfeld, Dorian.
Conjectures on elliptic curves over quadratic fields.
\emph{Number theory, Carbondale 1979 (Proc. Southern Illinois Conf., Southern Illinois Univ., Carbondale, Ill., 1979),} pp. 108–118, Lecture Notes in Math., 751, \emph{Springer, Berlin,} 1979.

\bibitem{GH85}
Goldfeld, Dorian; Hoffstein, Jeffrey.
Eisenstein series of $\frac{1}{2}$-integral weight and the mean value of real Dirichlet $L$-series.
\emph{Invent. Math.} 80 (1985), no. 2, 185--208.

\bibitem{Ha13}
Harper, Adam J.
Sharp conditional bounds for moments of the Riemann zeta function.
(Preprint).
arXiv:1305.4618

\bibitem{HL23}
Heap, Winston; Li, Junxian.
Simultaneous extreme values of zeta and $L$-functions.
\emph{Math. Ann.} 390 (2024), no. 4, 6355--6397.

\bibitem{HRS19}
Heap, Winston; Radziwi{\l \l}, Maksym; Soundararajan, Kannan.
Sharp upper bounds for fractional moments of the Riemann zeta function.
\emph{Q. J. Math.} 70 (2019), no. 4, 1387--1396.

\bibitem{HS22}
Heap, Winston; Soundararajan, Kannan.
Lower bounds for moments of zeta and $L$-functions revisited.
\emph{Mathematika} 68 (2022), no. 1, 1--14.

\bibitem{HB81}
Heath-Brown, D. R.
Fractional moments of the Riemann zeta function.
\emph{J. London Math. Soc.} (2) 24 (1981), no. 1, 65--78.

\bibitem{HB95}
Heath-Brown, D. R.
A mean value estimate for real character sums.
\emph{Acta Arith.} 72 (1995), no. 3, 235--275.

\bibitem{HL99}
Hoffstein, Jeffrey; Lockhart, Paul.
Omega results for automorphic $L$-functions. \emph{Automorphic forms, automorphic representations, and arithmetic (Fort Worth, TX, 1996),} 239–250, Proc. Sympos. Pure Math., 66, Part 2, \emph{Amer. Math. Soc., Providence, RI,} 1999.

\bibitem{HH22a}
Hua, Shenghao; Huang, Bingrong.
Determination of $\GL(3)$ cusp forms by central values of quadratic twisted  $L$-functions.
\emph{Int. Math. Res. Not. IMRN} 2023, no. 9, 7976--8007.

\bibitem{HH22b}
Hua, Shenghao; Huang, Bingrong.
Lower bounds for moments of quadratic twisted self-dual $\GL(3)$ central $L$-values.
\emph{Acta Math. Sin. (Engl. Ser.)} 39 (2023), no. 11, 2139--2148.

\bibitem{HH22c}
Hua, Shenghao; Huang, Bingrong.
Extreme central $L$-values of almost prime quadratic twists of elliptic curves.
\emph{Sci. China Math.} 66 (2023), no. 12, 2755--2766.

\bibitem{IK04}
Iwaniec, Henryk; Kowalski, Emmanuel.
Analytic number theory.
American Mathematical Society Colloquium Publications, 53.
\emph{American Mathematical Society, Providence, RI,} 2004.

\bibitem{Ju81}
Jutila, Matti Ilmari.
On the mean value of $L(\frac{1}{2},\chi)$ for real characters.
\emph{Analysis} 1 (1981), no. 2, 149--161.

\bibitem{KS99}
Katz, Nicholas M.; Sarnak, Peter.
Random matrices, Frobenius eigenvalues, and monodromy.
American Mathematical Society Colloquium Publications, 45.
\emph{American Mathematical Society, Providence, RI,} 1999.

\bibitem{KS00a}
Keating, Jonathan P.; Snaith, Nina Claire.
Random matrix theory and $\zeta(1/2+it)$. \emph{Comm. Math. Phys.} 214 (2000), no. 1, 57--89.

\bibitem{KS00b}
Keating, Jonathan P.; Snaith, Nina Claire.
Random matrix theory and $L$-functions at $s=1/2$.
\emph{Comm. Math. Phys.} 214 (2000), no. 1, 91--110.

\bibitem{Li22}
Li, Xiannan.
Moments of quadratic twists of modular $L$-functions.
\emph{Invent. Math.} 237 (2024), no. 2, 697--733.

\bibitem{MT14}
Milinovich, Micah B.; Turnage-Butterbaugh, Caroline L.
Moments of products of automorphic $L$-functions.
\emph{J. Number Theory} 139 (2014), 175--204.

\bibitem{Mo69}
Montgomery, Hugh L.
Zeros of $L$-functions.
\emph{Invent. Math.} 8 (1969), 346--354.

\bibitem{Mo71}
Montgomery, Hugh L.
Topics in multiplicative number theory. Lecture Notes in Mathematics, Vol. 227. \emph{Springer-Verlag, Berlin-New York,} 1971.

\bibitem{MV07}
Montgomery, Hugh L.; Vaughan, Robert C.
Multiplicative number theory. I. Classical theory.
Cambridge Studies in Advanced Mathematics, 97. \emph{Cambridge University Press, Cambridge,} 2007.

\bibitem{RS13}
Radziwi{\l \l}, Maksym; Soundararajan, Kannan. Continuous lower bounds for moments of zeta and $L$-functions.
\emph{Mathematika} 59 (2013), no. 1, 119--128.

\bibitem{RS15}
Radziwi{\l \l}, Maksym; Soundararajan, Kannan.
Moments and distribution of central L-values of quadratic twists of elliptic curves.
\emph{Invent. Math.} 202 (2015), no. 3, 1029--1068.

\bibitem{Ra78}
Ramachandra, Kanakanahalli.
Some remarks on the mean value of the Riemann zeta function and other Dirichlet series. I.
\emph{Hardy-Ramanujan J.} 1 (1978), 15 pp.

\bibitem{Ra80}
Ramachandra, Kanakanahalli.
Some remarks on the mean value of the Riemann zeta function and other Dirichlet series. II.
\emph{Hardy-Ramanujan J.} 3 (1980), 1--24.

\bibitem{RS06}
Rudnick, Ze\'ev; Soundararajan, Kannan.
Lower bounds for moments of $L$-functions: symplectic and orthogonal examples.
\emph{Multiple Dirichlet series, automorphic forms, and analytic number theory,} 293–303, Proc. Sympos. Pure Math., 75, \emph{Amer. Math. Soc., Providence, RI,} 2006.

\bibitem{Sh22}
Shen, Quanli.
The first moment of quadratic twists of modular $L$-functions.
\emph{Acta Arith.} 206 (2022), no. 4, 313--337.

\bibitem{Sono16}
Sono, Keiju.
Moments of the products of quadratic twists of automorphic $L$-functions.
\emph{Manuscripta Math.} 150 (2016), no. 3-4, 547--569.

\bibitem{Sono20}
Sono, Keiju.
The second moment of quadratic Dirichlet $L$-functions.
\emph{J. Number Theory} 206 (2020), 194--230.

\bibitem{So00}
Soundararajan, Kannan.
Nonvanishing of quadratic Dirichlet $L$-functions at $s=1/2$.
\emph{ Ann. of Math. (2)} 152 (2000), no. 2, 447--488.

\bibitem{So08}
Soundararajan, Kannan.
Extreme values of zeta and $L$-functions.
\emph{Math. Ann.} 342 (2008), no. 2, 467--486.

\bibitem{So09}
Soundararajan, Kannan.
Moments of the Riemann zeta function.
\emph{Ann. of Math.} (2) 170 (2009), no. 2, 981--993.

\bibitem{So21}
Soundararajan, Kannan.
The distribution of values of zeta and $L$-functions.
\emph{ICM--International Congress of Mathematicians. Vol. 2. Plenary lectures,} 1260--1310, \emph{EMS Press, Berlin,} 2023.

\bibitem{SY10}
Soundararajan, Kannan; Young, Matthew P.
The second moment of quadratic twists of modular $L$-functions.
\emph{J. Eur. Math. Soc. (JEMS)} 12 (2010), no. 5, 1097--1116.

\bibitem{Ti86}
Titchmarsh, Edward Charles.
The theory of the Riemann zeta-function. 
\emph{The Clarendon Press, Oxford University Press, New York,} 1986.

\bibitem{VT81}
Vinogradov, A. I.; Takhtadzhyan, L. A. Analogues of the Vinogradov-Gauss formula on the critical line.
Differential geometry, Lie groups and mechanics, IV.
\emph{Zap. Nauchn. Sem. Leningrad. Otdel. Mat. Inst. Steklov. (LOMI)} 109 (1981), 41--82, 180--181, 182--183.

\bibitem{Yo09}
Young, Matthew P.
The first moment of quadratic Dirichlet $L$-functions.
\emph{Acta Arith.} 138 (2009), no. 1, 73--99.

\bibitem{Yo13}
Young, Matthew P.
The third moment of quadratic Dirichlet $L$-functions.
\emph{Selecta Math. (N.S.)} 19 (2013), no. 2, 509--543.

\end{thebibliography}
\end{document}